\documentclass[12pt]{amsart}

\usepackage{ucs}

\usepackage{amssymb}
\usepackage{amsthm}
\usepackage{amsmath}
\usepackage{latexsym}
\usepackage[cp1251]{inputenc}
\usepackage{graphicx}
\usepackage{wrapfig}
\usepackage{caption}
\usepackage{subcaption}
\usepackage{indentfirst}
\usepackage[left=2.5cm,right=2.5cm,top=2.5cm,bottom=2.5cm,bindingoffset=0cm]{geometry}
\usepackage{enumerate}
\usepackage{makecell}
\usepackage{float}
\usepackage[normalem]{ulem}
\usepackage{color}

\def\O{\mathbf{O}}

\DeclareMathOperator{\aut}{Aut}

\DeclareMathOperator{\cay}{Cay}
\DeclareMathOperator{\cyc}{Cyc}
\DeclareMathOperator{\dev}{Dev}

\DeclareMathOperator{\GaL}{{\rm \Gamma}L}
\DeclareMathOperator{\PGaL}{{\rm P\Gamma}L}

\DeclareMathOperator{\GL}{GL}

\DeclareMathOperator{\iso}{Iso}

\DeclareMathOperator{\orb}{Orb}

\DeclareMathOperator{\PSL}{PSL}
\DeclareMathOperator{\PGL}{PGL}

\DeclareMathOperator{\rk}{rk}

\DeclareMathOperator{\SL}{SL}

\DeclareMathOperator{\Span}{Span}

\DeclareMathOperator{\sym}{Sym}

\DeclareMathOperator{\Hol}{Hol}

\def\r{\mathrm{right}}

\def\tm#1{\item[{\rm (#1)}]}

\makeatletter 
\def\@seccntformat#1{\csname the#1\endcsname. } 
\def\@biblabel#1{#1.}

\makeatother

\title[DSRGs and DDGs from Tatra association schemes]{Directed strongly regular graphs and divisible design graphs from Tatra association schemes}

\author{Mikhail Muzychuk}
\address{Ben Gurion University of the Negev, Beer Sheva, Israel}
\email{muzychuk@bgu.ac.il}

\author{Grigory Ryabov}
\address{Sobolev Institute of Mathematics, Novosibirsk, Russia}
\email{gric2ryabov@gmail.com}

\thanks{}

\date{}

\newtheorem{prop}{Proposition}[section]

\newtheorem{lemm}[prop]{Lemma}
\newtheorem{theo}[prop]{Theorem}
\newtheorem*{ques}{Question}
\newtheorem{corl}[prop]{Corollary}

\theoremstyle{definition}

\newtheorem*{rem1}{Remark~1}
\newtheorem*{rem2}{Remark~2}
\newtheorem*{rem3}{Remark~3}

\begin{document}

\begin{abstract}
In this paper, we construct directed strongly regular graphs and divisible design graphs with new parameters merging some basic relations of so-called Tatra associations schemes. We also study the above association schemes, their fusions and isomorphisms.
\end{abstract}

\maketitle

\section{Introduction}
 
A \emph{strongly regular graph} (\emph{SRG} for short) with parameters $(v,k,\lambda,\mu)$ is defined to be a $k$-regular graph on $v$ vertices such that the number of common neighbors of any two distinct vertices $\alpha$ and $\beta$ is equal to $\lambda$ or $\mu$ when $(\alpha,\beta)$ is an edge or $(\alpha,\beta)$ is not an edge, respectively. Nowadays, SRGs are realized as one of the crucial objects in algebraic combinatorics. For a background of SRGs, we refer the readers to the monograph~\cite{BM}. 

The following generalization of the notion of an SRG was introduced in~\cite{Duv}. A \emph{directed strongly regular graph} (\emph{DSRG} for short) with parameters $(v,k,t,\lambda,\mu)$ is a directed graph on $v$ vertices such that every vertex has indegree and outdegree~$k$ and for any two vertices $\alpha,\beta$ the number of directed paths of the form $\alpha\rightarrow \gamma \rightarrow  \beta$ is equal to $t$, $\lambda$, or $\mu$ when $\alpha=\beta$, $(\alpha,\beta)$ is an arc, or $(\alpha,\beta)$ is not an arc, respectively. The feasible parameters of DSRGs and information  and bibliography on them are collected on the web-page~\cite{BH}. 

Several DSRGs arising from the action of linear groups were found recently in~\cite{BCSZ}. Two nonisomorphic DSRGs with parameters $(63,11,8,1,2)$ were constructed in~\cite{BCS}. There is the following remark in that paper: ``The referee asked whether this can be generalized. But no, this is a sporadic situation''.  In the present paper, we construct a family of DSRGs containing two above mentioned graphs.

\begin{theo}\label{main1}
Let $p$ be a prime and $q$ a prime power such that 
$$p \equiv 3 \mod~4~\text{and}~q-1=p(p-3)/4.$$
Then there exist two nonisomorphic directed strongly regular graphs with parameters
$$\left((q+1)p,q+\frac{p-1}{2},q,\frac{q-1}{p},\frac{q-1}{p}+1\right)$$
and automorphism group isomorphic to $\PGaL(2,q)$.
\end{theo}

\begin{rem1}
One can see that the number of pairs $(p,q)$ of primes such that $p \equiv 3 \mod~4$ and $q-1=p(p-3)/4$ is equal to the number of positive integers~$t$ such that both numbers $4t+3$ and $(4t+3)t+1$ are primes. The set consisting of all such~$t$ is infinite modulo the Bateman-Horn conjecture (see, e.g.~\cite{BH,JZ}, for details). Therefore the family of DSRGs from Theorem~\ref{main1} is infinite modulo the Bateman-Horn conjecture.
\end{rem1}

The smallest pair $(p,q)$ satisfying the conditions of Theorem~\ref{main1} is $(7,8)$. One can verify that this pair leads to two DSRGs with parameters $(63,11,8,1,2)$ from~\cite{BCS}. It can be checked by a computer calculation that there are exactly $328$ pairs of primes $(p,q)$ satisfying the condition of Theorem~\ref{main1} with $q\leq 10^9$. The first three of them are $(11,23)$, $(43,431)$, $(59,827)$.

The second result of the paper is concerned with divisible design graphs. A \emph{divisible design graph} (\emph{DDG} for short) with parameters $(v,k,\lambda_1,\lambda_2,m,n)$ is a $k$-regular graph on $v$ vertices whose vertex set can be partitioned into~$m$ classes of size~$n$ such that any two distinct vertices from the same class have exactly~$\lambda_1$ common neighbors and any two vertices from different classes have exactly~$\lambda_2$ common neighbors. The notion of a DDG was introduced in~\cite{HKM} as a generalization of the notion of a $(v,k,\lambda)$-graph which is an SRG with $\lambda=\mu$ or, equivalently, a DDG with $m=1$, or $n=1$, or $\lambda_1=\lambda_2$. A DDG is said to be \emph{proper} if $m,n>1$ and $\lambda_1\neq \lambda_2$. DDGs are precisely those graphs whose adjacency matrices are incidence matrices of symmetric divisible designs. For an information on DDGs, we refer the readers to~\cite{HKM}.

We provide a new family of DDGs. One of the key ingredients of the construction is a difference set in a cyclic group. Recall that a nonempty subset $D$ of a finite group $G$ is called a \emph{difference set} (\emph{DS} for short) in $G$ if there exists a nonnegative integer $\lambda$ such that every nontrivial element of $G$ has exactly $\lambda$ representations in the form $g_1g_2^{-1}$, where $g_1,g_2\in D$ (see~\cite[Chapter~VI]{BJL}). A DS $D$ is said to be \emph{nontrivial} if $2\leq |D|\leq |G|-2$ and \emph{trivial} otherwise. The above mentioned family of DDGs is given in the theorem below.

\begin{theo}\label{main2}
Let $q$ be a prime power and $n\geq 1$ a divisor of $q-1$ such that $q(q-1)/n$ is even. Suppose that a cyclic group of order~$n$ has a difference set (possibly, trivial) with parameters $(n,k,\lambda)$. Then there exists a divisible design graph with parameters
$$\left(n(q+1),kq,\lambda q,\frac{k^2(q-1)}{n},q+1,n\right).$$
\end{theo}

The graphs from Theorems~\ref{main1} and~\ref{main2} are constructed as fusions of one class of association schemes, namely, so-called \emph{Tatra schemes}. This class of schemes arises from a symmetric bilinear form defined on the equivalence classes of nonzero $2$-dimensional vectors modulo some subgroup of the multiplicative group of a finite field. Tatra schemes were introduced in~\cite{Reich} and generalize the construction from~\cite{DF}. In fact, the equivalent construction was described in other terms in~\cite[Section~4]{KhS}. A significant part of the paper is devoted to studying of Tatra schemes. Namely, we
\begin{enumerate}

\tm{1} describe a family of noncommutative schemes which are fusions of Tatra schemes (Theorem~\ref{fuse});

\tm{2} find the automorphism and isomorphism groups of Tatra schemes (Theorem~\ref{mainautiso}).

\end{enumerate} 

We finish the introduction with a brief outline of the paper. Section~$2$ contains a necessary background of association schemes, $S$-rings, graphs, and difference sets. In Section~$3$, we describe the construction of Tatra schemes and study their fusions and isomorphisms. In Sections~$4$ and~$5$, we prove Theorems~\ref{main1} and~\ref{main2}, respectively.

\vspace{5mm}

\noindent \textbf{Notation.} 

\vspace{2mm}

\noindent If $K\leq \sym(\Omega)$, $\alpha\in \Omega$, and $\Lambda\subseteq \Omega$, then the one-point stabilizer of $\alpha$ in $K$ and the setwise stabilizer of $\Lambda$ in $K$ are denoted by $K_{\alpha}$ and $K_{\{\Lambda\}}$, respectively. 

\vspace{2mm}

\noindent If $\Lambda$ is a $K$-invariant subset of $\Omega$ and $\mathcal{L}$ is an imprimitivity system for $K$, then the permutation groups induced by $K$ on $\Lambda$ and $\mathcal{L}$ are denoted by $K^\Lambda$ and $K^\mathcal{L}$, respectively. If $f\in K$, then the permutations induced by $f$ on $\Lambda$ and $\mathcal{L}$ are denoted by $f^\Lambda$ and $f^\mathcal{L}$, respectively. 

\vspace{2mm}

\noindent  For short, put $(K_{\{\Lambda\}})^\Lambda=K_{\{\Lambda\}}^\Lambda$, $(K_{\{\Lambda\}})^{\mathcal{L}}=K_{\{\Lambda\}}^{\mathcal{L}}$, $(K_{\alpha})^{\mathcal{L}}=K_{\alpha}^{\mathcal{L}}$, where $\mathcal{L}$ is an imprimitivity system for $K$, $\Lambda\in \mathcal{L}$, and $\alpha\in \Omega$.

\vspace{2mm}

\noindent The set of all orbits of $K$ on $\Omega$ is denoted by $\orb(K,\Omega)$.

\vspace{2mm}

\noindent If $H\leq G$, then the centralizer and normalizer of $H$ in $G$ are denoted by $C_G(H)$ and $N_G(H)$, respectively. The center of a group $G$ is denoted by $Z(G)$.  

\vspace{2mm}

\noindent The general linear, projective general linear, general semilinear, projective general semilinear, special linear, projective special linear groups of dimension~$n$ over a field of order~$q$ are denoted by $\GL(n,q)$, $\PGL(n,q)$, $\GaL(n,q)$, $\PGaL(n,q)$, $\SL(n,q)$, $\PSL(n,q)$, respectively. 

\vspace{2mm}

\noindent The Cayley digraph with connection set $S$ is denoted by $\cay(G,S)$.

\vspace{2mm}

\noindent Given a positive integer $n$, the identity and all-identity matrices of size $(n\times n)$ are denoted by $I_n$ and $J_n$, respectively. If $n$ is clear from the context, then we write just $I$ or $J$. 

\vspace{2mm}

\noindent Given a binary relation $s\subseteq \Omega^2$, the adjacency matrix of~$s$ is denoted by $A(s)$.

\section{Preliminaries}
In this section, we provide a necessary background of association schemes, $S$-rings, graphs, and difference sets.

\subsection{Association schemes}
We follow~\cite{CP} in the presentation of the material on association schemes.  

Let $\Omega$ be a finite set. The diagonal of $\Omega^2$ is denoted by $1_{\Omega}$. Given a binary relation $s\subseteq \Omega^2$ and $\alpha\in \Omega$, put $\alpha s=\{\beta\in \Omega:~(\alpha,\beta)\in s\}$ and $s^*=\{(\gamma,\beta):~(\beta,\gamma)\in s\}$. Given a collection $S$ of binary relations on $\Omega^2$, put $S^*=\{s^*:~s\in S\}$. Given $r,s\subseteq \Omega^2$, a composition of $r$ and $s$ is denoted by $rs$. 

Let $S$ be a partition of $\Omega^2$. The pair $\mathcal{X}=(\Omega,S)$ is called an \emph{association scheme} or just a \emph{scheme} on $\Omega$ if $1_{\Omega}\in S$, $S^{*}=S$, and given $r,s,t\in S$ the number
$$c_{rs}^t=|\alpha r\cap \beta s^{*}|$$
does not depend on the choice of $(\alpha,\beta)\in t$. The elements of $S$, the numbers $c_{rs}^t$, and the number $\rk(\mathcal{X})=|S|$ are called the \emph{basic relations}, the \emph{intersection numbers}, and the \emph{rank} of $\mathcal{X}$, respectively. It is easy to verify that $(\Omega,\{1_\Omega,\Omega^2\setminus 1_{\Omega}\})$ is a scheme. Such scheme is said to be \emph{trivial}.

A scheme is said to be \emph{commutative} if $c_{rs}^t=c_{sr}^t$ for all $r,s,t\in S$ and \emph{noncommutative} otherwise. Given $r\in S$, the number $n_r=c_{rr^*}^{\textbf{1}_\Omega}$ is called the \emph{valency} of $r$. It is easy to see that $n_r=|\alpha r|$ for every $\alpha\in \Omega$. An equivalence relation on $\Omega$ is defined to be a \emph{parabolic} of $\mathcal{X}$ if it is a union of some basic relations of $\mathcal{X}$. The set $\O_{\theta}(S)$ of all basic relations of $\mathcal{X}$ of valency~$1$ form a group with respect to a composition called the \emph{thin radical} of $\mathcal{X}$. The union of all basic relations from $\O_{\theta}(S)$ is a parabolic of $\mathcal{X}$ called the \emph{thin radical parabolic}.

Let $\mathcal{X}=(\Omega,S)$ and $\mathcal{X}^{\prime}=(\Omega^{\prime},S^{\prime})$ be schemes. An \emph{isomorphism} from $\mathcal{X}$ to $\mathcal{X}^{\prime}$ is defined to be a bijection $f:\Omega\rightarrow \Omega^{\prime}$ such that $S^{\prime}=S^f$, where $S^f=\{s^f:~s\in S\}$ and $s^f=\{(\alpha^f,\beta^f):~(\alpha,~\beta)\in s\}$. The group $\iso(\mathcal{X})$ of all isomorphisms from $\mathcal{X}$ onto itself has a normal subgroup
$$\aut(\mathcal{X})=\{f\in \iso(\mathcal{X}): s^f=s~\text{for every}~s\in S\}$$
called the \emph{automorphism group} of $\mathcal{X}$. If $\mathcal{X}$ is trivial, then $\aut(\mathcal{X})=\sym(\Omega)$.

The set of all equivalence classes of each parabolic of $\mathcal{X}$ form an imprimitivity system for $\aut(\mathcal{X})$. If $\Lambda$ is a class of the thin radical parabolic, then 
\begin{equation}\label{thinrad0}
\aut(\mathcal{X})_{\{\Lambda\}}^{\Lambda}~\text{is a regular group isomorphic to}~\O_{\theta}(S).
\end{equation}

A scheme $\mathcal{X}=(\Omega,S)$ is said to be \emph{schurian} if there exists a permutation group $K\leq \sym(\Omega)$ such that $S=\orb(K,\Omega^2)$, where $K$ acts on $\Omega^2$ componentwise. One can verify that $\mathcal{X}$ is schurian if and only if $S=\orb(\aut(\mathcal{X}),\Omega^2)$. 

One can define a partial order on the set of all schemes on $\Omega$. Namely, put $\mathcal{X}\leq \mathcal{X}^\prime$ if every basic relation of $\mathcal{X}$ is a union of some basic relations of $\mathcal{X}^\prime$. In this case, we say that $\mathcal{X}$ is a \emph{fusion} of $\mathcal{X}^\prime$. The trivial scheme is the smallest element with respect to this order. 


Let $\mathcal{X}=(\Omega,S)$ be a scheme. The linear space $\mathcal{M}=\mathcal{M}(\mathcal{X})=\Span_{\mathbb{C}}\{A(s):~s\in S\}$ is an algebra with respect to a matrix multiplication and the intersection numbers of $\mathcal{X}$ are the structure constants of $\mathcal{M}$ with respect to the basis $\{A(s):~s\in S\}$. 

Conversely, if $S$ is a partition of $\Omega^2$ closed under taking the inverse relation, $1_\Omega\in S$, and $\Span_{\mathbb{C}}\{A(s):~s\in S\}$ is an algebra with respect to a matrix multiplication, then $\mathcal{X}=(\Omega,S)$ is a scheme.

\subsection{$S$-rings}
In this subsection, we provide a necessary background of $S$-rings which are special subrings of the group ring introduced in the classical papers of Schur and Wielandt~\cite{Schur,Wi}. We use the notations and terminology from~\cite[Section~2.4]{CP} and~\cite{Ry}.

Let $G$ be a finite group and $\mathbb{Z}G$ the group ring over the integers. The identity element of $G$ and the set of all nonidentity elements of $G$ are denoted by $e$ and $G^\#$, respectively. Let $\xi=\sum_{g\in G} a_g g,\eta=\sum_{g\in G} b_g g\in\mathbb{Z}G$. The product of $\xi$ and $\eta$ will be written as $\xi\cdot \eta$. The element $\sum_{g\in G} a_g g^{-1}$ is denoted by $\xi^{(-1)}$. The number $\sum_{g\in G} a_g$ is denoted by $\varepsilon(\xi)$. Given $X\subseteq G$, we set 
$$\underline{X}=\sum \limits_{x\in X} {x}\in\mathbb{Z}G~\text{and}~X^{(-1)}=\{x^{-1}:x\in X\}.$$

A subring  $\mathcal{A}\subseteq \mathbb{Z}G$ is called an \emph{$S$-ring} (a \emph{Schur ring}) over $G$ if there exists a partition $\mathcal{S}=\mathcal{S}(\mathcal{A})$ of~$G$ such that:
\begin{enumerate}

\tm{1} $\{e\}\in\mathcal{S}$;

\tm{2} if $X\in\mathcal{S}$, then $X^{(-1)}\in\mathcal{S}$;

\tm{3} $\mathcal{A}=\Span_{\mathbb{Z}}\{\underline{X}:~X\in\mathcal{S}\}$.
\end{enumerate}

\noindent  The elements of $\mathcal{S}$ are called the \emph{basic sets} of $\mathcal{A}$ and the number $\rk(\mathcal{A})=|\mathcal{S}|$ is called the \emph{rank} of~$\mathcal{A}$. 

If $X\subseteq G$ and $\underline{X}\in \mathcal{A}$, then $X$ is defined to be an \emph{$\mathcal{A}$-set}. A subgroup $A$ of $G$ such that $\underline{A}\in \mathcal{A}$ is called an \emph{$\mathcal{A}$-subgroup}. If $A$ is an $\mathcal{A}$-subgroup, then put $\mathcal{S}(\mathcal{A})_A=\{X\in \mathcal{S}(\mathcal{A}):~X\subseteq A\}$.

For every $M\leq \aut(G)$, the partition of $G$ into the orbits of $M$ defines the $S$-ring $\mathcal{A}$ over~$G$. In this case, $\mathcal{A}$ is called \emph{cyclotomic} and denoted by $\cyc(M,G)$.

The \emph{automorphism group} $\aut(\mathcal{A})$ of $\mathcal{A}$ is defined to be the group 
$$\bigcap \limits_{X\in \mathcal{S}(\mathcal{A})} \aut(\cay(G,X)).$$

\subsection{Graphs}
In this subsection, we recall two criteria for a graph to be SRG or DDG. The first of them can be found in~\cite[Eqs.~1.4-1.5]{Duv}, whereas the second one in~\cite[Eq.~1]{HKM}.

\begin{lemm}\label{dsrg0}
A digraph $\Gamma$ is a directed strongly regular graph with parameters $(v,k,t,\lambda,\mu)$ if and only if the adjacency matrix $A$ of $\Gamma$ satisfies the following equalities:
\begin{enumerate}

\tm{1} $AJ_v=J_vA=kJ_v$,

\tm{2} $A^2=tI_v+\lambda A+\mu(J_v-I_v-A)$.

\end{enumerate}
Moreover, $\Gamma$ is a strongly regular graph if and only if $t=k$.
\end{lemm}

\begin{lemm}\label{ddg0}
A graph $\Gamma$ is a divisible design graph with parameters $(v,k,\lambda_1,\lambda_2,m,n)$ if and only if the adjacency matrix $A$ of $\Gamma$ satisfies the following equality:
$$A^2=kI_v+\lambda_1(I_m\otimes J_n-I_v)+\lambda_2(J_v-I_m\otimes J_n).$$
\end{lemm}

\subsection{Difference sets}
A nonempty subset $D$ of $G$ is called a \emph{difference set} (\emph{DS} for short) in $G$ if 
$$\underline{D}\cdot\underline{D}^{-1}=ke+\lambda \underline{G}^\#,$$
where $k=|D|$ and $\lambda$ is a nonnegative integer. The numbers $(v,k,\lambda)$, where $v=|G|$, are called the \emph{parameters} of $D$. It is easy to check that $G\setminus D$ is a DS with parameters $(v,v-k,v-2k+\lambda)$. The set~$G$ and every subset of $G$ of size~$1$ or~$|G|-1$ are DSs. The DS $D$ is said to be \emph{nontrivial} if $2\leq|D|\leq |G|-2$ and \emph{trivial} otherwise. For the general theory of DSs, we refer the readers to~\cite[Chapter~VI]{BJL}.

Two DSs $D_1$ and $D_2$ in groups $G_1$ and $G_2$, respectively, are said to be \emph{equivalent} if there exist an isomorphism $\varphi$ from $G_1$ to $G_2$ and $g\in G_2$ such that $D_1^\varphi=D_2g$. If $D$ is a DS, then the incidence structure $\dev(D)=(G,\mathcal{B})$ with point set~$G$ and block set~$\mathcal{B}=\{Dg:~g\in G\}$ is a symmetric block design. The designs $\dev(D_1)=(G_1,\mathcal{B}_1)$ and $\dev(D_2)=(G_2,\mathcal{B}_2)$ are said to be \emph{isomorphic} if there exist bijections $\varphi_1:G_1\rightarrow G_2$ and $\varphi_2:\mathcal{B}_1\rightarrow\mathcal{B}_2$ such that 
$$g\in B\Leftrightarrow g^{\varphi_1}\in B^{\varphi_2}$$
for all $g\in G_1$ and $B\in \mathcal{B}_1$. It is easy to verify that if $D_1$ and $D_2$ are equivalent, then $\dev(D_1)$ and $\dev(D_2)$ are isomorphic. The converse statement does not hold in general, however it holds in case of a cyclic group of prime order. This follows from the lemma below which is a corollary of~\cite[Theorem~A]{Palfy}.

\begin{lemm}\label{palfy}
Let $D_1$ and $D_2$ be difference sets in a cyclic group of prime order. Then $\dev(D_1)$ and $\dev(D_2)$ are isomorphic if and only if $D_1$ and $D_2$ are equivalent.
\end{lemm}

\section{Tatra schemes}

In this section, we recall the construction and study some properties of so-called Tatra schemes from~\cite{Reich} which is the key ingredient in the proofs of the main results. It should be noted that the equivalent construction was given independently in~\cite{KhS}. Noncommutative fusions of rank~$6$ of such schemes in a special case were studied in~\cite{DF}.

\subsection{Construction}

Let $r$ be a prime, $q=r^d$ for some $d\geq 1$, and $\mathbb{F}=\mathbb{F}_q$ a field of order~$q$. The multiplicative group of $\mathbb{F}$ is denoted by $\mathbb{F}^*$. Let $n$ be a divisor of $q-1$ such that
\begin{equation}\label{maineq} 
\frac{q(q-1)}{n}~\text{is even}, 
\end{equation}
$K$ a subgroup of $\mathbb{F}^*$ of index~$n$, and $C=\mathbb{F}^*/K$. Clearly, $C\cong C_n$ and $m=|K|=(q-1)/n$ is even whenever $q$ is odd. We denote the identity element of $C$ by~$e$. If $x\in\mathbb{F}$ and $g=Ky\in C$, then put $xg=K(xy)$. It should be mentioned that $q$ and $n$ are assumed to be odd in~\cite{Reich}. However, the only essential condition for the construction is that $(-1)\in K$ or, equivalently, $m$ is even if $q$ is odd. The latter is guaranteed by Eq.~\eqref{maineq}.

Let $V$ be a $2$-dimensional vector space over $\mathbb{F}$ and 
$$\Omega=\{Kv:~v\in V\setminus \{0\}\},$$
where $Kv=\{xv:~x\in K\}$. Given $\alpha=Kv\in \Omega$, put 
$$C\alpha=\{g\alpha:~g\in C\}=\{K(xv):~x\in \mathbb{F}\}.$$
One can see that $|\Omega|=|V\setminus \{0\}|/|K|=(q^2-1)/m=(q+1)n$. Let
$$\langle\cdot,~\cdot \rangle: \Omega\times \Omega \rightarrow C\cup \{0\}$$
be the form defined as follows:
$$\langle Ku,Kv\rangle=K\det(u,v),$$
where for the vectors $u=(u_1,u_2)^T$ and $v=(v_1,v_2)^T$, 
$$\det(u,v)=u_1v_2-u_2v_1.$$
It is easy to see that $\langle Ku,Kv\rangle=0$ if and only if $Kv=gKu$ for some $g\in C$. The form $\langle\cdot,~\cdot \rangle$ is well-defined by~\cite[Lemma~3]{Reich}. It is also symmetric. Indeed, this is clear if $q$ is even. If $q$ is odd, then $m$ is even and hence $(-1)\in K$ which implies the symmetry of $\langle\cdot,~\cdot \rangle$.

Given $g\in C$, let us define two binary relations $r_g$ and $s_g$ on $\Omega^2$ as follows:
$$r_g=\{(\alpha,\beta)\in \Omega^2:~\langle \alpha,\beta \rangle=0,~\beta=g\alpha\},$$
$$s_g=\{(\alpha,\beta)\in \Omega^2:~\langle \alpha,\beta \rangle=g\}.$$
Put $A_g=A(r_g)$ and $B_g=A(s_g)$. Until the end of the paper, we assume that $I=I_{(q+1)n}$ and $J=J_{(q+1)n}$. Observe that $r_e=1_{\Omega}$ and hence $A_e=I$.

\begin{prop}\cite[Theorem~3]{Reich}\label{scheme}
With the above notation, the pair $\mathcal{X}_0=(\Omega,S_0)$, where $S_0=\{r_g,s_g:~g\in C\}$, is a scheme of rank~$2n$. Moreover, 
\begin{enumerate}
\tm{1} $A_hA_g=A_gA_h=A_{hg}$,

\tm{2} $A_hB_g=B_{h^{-1}g}$, $B_gA_h=B_{gh}$,

\tm{3} $B_hB_g=qA_{h^{-1}g}+m(J-A_C)$,

\tm{4} $n_{r_g}=1$, $n_{s_g}=q$,

\tm{5} $r_g^*=r_{g^{-1}}$, $s_g^*=s_g$
\end{enumerate}
for all $h,g\in C$.
\end{prop}

Note that in~\cite[Theorem~3]{Reich}, there is a misprint. Namely, the right-hand sides of the equalities from the last line on p.~223 and the first line on p.~224 should be interchanged. 

Proposition~\ref{scheme} implies the following corollary.

\begin{corl}\label{noncommut}
With the above notation, $\mathcal{X}_0$ is noncommutative if and only if $n\geq 3$.
\end{corl}

If $n=1$, then $\mathcal{X}_0$ is trivial, whereas if $n=2$, then $\mathcal{X}_0$ is a commutative scheme of rank~$4$. 

Let us define the binary operation $\star$ on $S_0$ as follows. If one of the relations $r,s\in S_0$ is of valency~$1$, then put $r\star s=rs$. Otherwise, $r=s_{h}$ and $s=s_{g}$ for some $h,g\in C$. In this case, put $r\star s=r_{h^{-1}g}$. Clearly, $r_e=1_\Omega$ is an identity element with respect to $\star$ and $s\star s^*=s^*\star s=1_{\Omega}$ for every $s\in S$. So Proposition~\ref{scheme} implies the following corollary.

\begin{corl}\label{dih}
The set $S_0$ equipped with a binary operation $\star$ is a group isomorphic to a dihedral group $D_{2n}$.
\end{corl}

Due to~\cite[Theorem~4]{Reich}, the graph $(\Omega,s_g)$ is a distance-regular antipodal graph of diameter~$3$ with intersection array~$(q,(n-1)m,1;1,m,q)$ for every $g\in G$. Further, we will show that all of these graphs are pairwise isomorphic (see Corollary~\ref{isodrg}). In fact, it is possible to verify that all of them are isomorphic to a graph from~\cite[Theorem~12.5.3]{BCN}.

\subsection{Fusions}

Given $\xi=\sum \limits_{g\in C} a_gg\in \mathbb{Z}C$, put $A_{\xi}=\sum \limits_{g\in C} a_gA_g$ and $B_{\xi}=\sum \limits_{g\in C} a_gB_g$. If $\xi=\underline{X}$ for some $X\subseteq G$, then we write $A_X$ and $B_X$ instead of $A_{\underline{X}}$ and $B_{\underline{X}}$, respectively. Given $X\subseteq G$, put
$$r_X=\bigcup \limits_{g\in X} r_g,~s_X=\bigcup \limits_{g\in X} s_g.$$
Clearly, $A_X=A(r_X)$ and $B_X=A(s_X)$.

Since $S$ is a partition of $\Omega^2$,
\begin{equation}\label{all1sum}
A_C+B_C=\sum \limits_{g\in C} A_g+\sum \limits_{g\in C} B_g=J.
\end{equation}
It follows easily from the definitions of $r_g$ and $r_C$ that
\begin{equation}\label{rc}
r_C=\{(\alpha,\beta)\in \Omega^2:~\langle \alpha,\beta \rangle=0\}.
\end{equation}
Statements~$(1)$ and~$(4)$ of Proposition~\ref{scheme} imply that $r_C$ is a thin radical parabolic of $\mathcal{X}_0$ of valency~$n=|C|$ and 
\begin{equation}\label{thinrad}
\O_{\theta}(S_0)=\{r_g:~g\in C\}\cong C. 
\end{equation}
So
\begin{equation}\label{parab}
A_C=I_{q+1}\otimes J_n.
\end{equation}

\begin{lemm}\label{aux}
With the above notation,
\begin{enumerate}
\tm{1} $A_{\xi}A_{\eta}=A_{\xi\cdot \eta}$,

\tm{2} $A_{\xi}B_{\eta}=B_{\xi^{(-1)}\cdot\eta}$, $B_{\eta}A_{\xi}=B_{\xi\cdot\eta}$,

\tm{3} $B_{\xi}B_{\eta}=qA_{\xi^{(-1)}\cdot \eta}+m\varepsilon(\xi^{(-1)}\cdot \eta)(J-A_C)$,

\end{enumerate}
for all $\xi,\eta \in \mathbb{Z}C$
\end{lemm}

\begin{proof}
Let us compute the required products. Suppose that $\xi=\sum \limits_{g\in C} a_gg$ and $\eta=\sum \limits_{g\in C} b_gg$, where $a_g,b_g\in \mathbb{Z}$ for every $g\in C$. Then 
$$A_{\xi}A_{\eta}=\left(\sum \limits_{h\in C} a_hA_h\right)\left(\sum \limits_{g\in C} b_gA_g\right)=\sum\limits_{h\in C}\sum\limits_{g\in C} a_hb_gA_hA_g=\sum\limits_{h\in C}\sum\limits_{g\in C} a_hb_gA_{hg}=A_{\xi\cdot \eta},$$
where the the third equality follows from Proposition~\ref{scheme}(1).
Furhter,
$$A_{\xi}B_{\eta}=\left(\sum \limits_{h\in C} a_hA_h\right)\left(\sum \limits_{g\in C} b_gB_g\right)=\sum\limits_{h\in C}\sum\limits_{g\in C} a_hb_gA_hB_g=\sum\limits_{h\in C}\sum\limits_{g\in C} a_hb_gB_{h^{-1}g}=B_{\xi^{(-1)}\cdot\eta},$$
where the third equality holds by the first equality from Proposition~\ref{scheme}(2). The equality $B_{\eta}A_{\xi}=B_{\xi\cdot\eta}$ can be verified in the same way using the second equality from Proposition~\ref{scheme}(2) instead of the first one. Finally,
$$B_{\xi}B_{\eta}=\left(\sum \limits_{g\in C} a_hB_h\right)\left(\sum \limits_{g\in C} b_gB_g\right)=\left(\sum \limits_{g\in C} a_hB_eA_h\right)\left(\sum \limits_{g\in C} b_gB_eA_g\right)=B_eA_{\xi}B_eA_{\eta}=$$
$$=B_e^2A_{\xi^{(-1)}\cdot\eta}=(qA_{e}+m(J-A_C))A_{\xi^{(-1)}\cdot\eta}=qA_{\xi^{(-1)}\cdot \eta}+m\varepsilon(\xi^{(-1)}\cdot \eta)(J-A_C),$$
where the second, fourth, and fifth equalities follow from Proposition~\ref{scheme}(2), Lemma~\ref{aux}(2), and Proposition~\ref{scheme}(3), respectively, whereas the last equality holds because $JA_g=n_{r_g}J=J$ (Proposition~\ref{scheme}(4)) and $A_CA_g=A_{Cg}=A_C$ (Lemma~\ref{aux}(1)) for every $g\in C$.
\end{proof}

Further, we generalize Theorems~$5$ and~$6$ from~\cite{Reich} and show that every $S$-ring $\mathcal{A}$ over a dihedral group of order $2n$ such that a cyclic subgroup of order~$n$ is an $\mathcal{A}$-subgroup induces a scheme which is a fusion of $\mathcal{X}_0$. Let $G=C\rtimes \langle b \rangle$, where $|b|=2$ and $g^b=g^{-1}$ for every $g\in C$. 

\begin{theo}\label{fuse}
Let $\mathcal{A}$ be an $S$-ring over $G$ such that $C$ is an $\mathcal{A}$-subgroup. Then the pair $\mathcal{X}_{\mathcal{A}}=(\Omega,S_{\mathcal{A}})$, where
$$S_{\mathcal{A}}=\{r_X,s_Y:~X,Y\subseteq C,~X,bY\in \mathcal{S}(\mathcal{A})\},$$
is a scheme.
\end{theo}

\begin{proof}
One can see that $1_{\Omega}=r_e\in S_{\mathcal{A}}$ because $\{e\}\in \mathcal{S}(\mathcal{A})_C$. Proposition~\ref{scheme}(5) implies that 
$$r_{X}^*=\bigcup\limits_{g\in X} r_g^*=\bigcup\limits_{g\in X} r_{g^{-1}}=r_{X^{(-1)}}$$ 
and 
$$s_Y^*=\bigcup\limits_{g\in Y} s_g^*=\bigcup\limits_{g\in Y} s_{g}=s_Y$$ 
for all $X,Y\subseteq C$. Since $X^{(-1)}\in \mathcal{S}(\mathcal{A})_C$ for every $X\in \mathcal{S}(\mathcal{A})_C$, we conclude that $S_{\mathcal{A}}$ is closed under taking the inverse relation. Note that
\begin{equation}\label{parab2}
A_C=\sum \limits_{g\in C} A_g =\sum \limits_{X\in \mathcal{S}(\mathcal{A}_C)} A_X.
\end{equation}

To finish the proof, it remains to verify that the linear span of matrices from the set
$$A(\mathcal{X}_{\mathcal{A}})=\{A_X,B_Y:~X,bY\in \mathcal{S}(\mathcal{A})\}$$
is closed under multiplication. Let $X_1,X_2,Y_1,Y_2\subseteq C$ be such that $X_1,X_2,bY_1,bY_2\in \mathcal{S}(\mathcal{A})$. Then
\begin{equation}\label{belong}
\underline{X_1}\cdot\underline{X_2},b\underline{X_1}^{(-1)}\cdot\underline{Y_1},b\underline{Y_1}\cdot\underline{X_1},\underline{Y_1}^{(-1)}\cdot \underline{Y_2}\in \mathcal{A}.
\end{equation}
Due to Lemma~\ref{aux}, we obtain
$$A_{X_1}A_{X_2}=A_{\underline{X_1}\cdot\underline{X_2}},$$
$$A_{X_1}B_{Y_1}=B_{\underline{X_1}^{(-1)}\cdot\underline{Y_1}},~B_{Y_1}A_{X_1}=B_{\underline{Y_1}\cdot\underline{X_1}},$$
$$B_{Y_1}B_{Y_2}=qA_{\underline{Y_1}^{(-1)}\cdot \underline{Y_2}}+m\varepsilon(\underline{Y_1}^{(-1)}\cdot \underline{Y_2})(J-A_C).$$
Eqs.~\eqref{parab2} and~\eqref{belong} yield that the right-hand side and hence the left-hand side of each of the above equalities are linear combinations of matrices from $A(\mathcal{X}_{\mathcal{A}})$ as required.
\end{proof}

If $\mathcal{A}$ is an $S$-ring of rank~$4$ constructed from a DS in a cyclic group, then $\mathcal{X}_{\mathcal{A}}$ is a scheme from~\cite[Theorem~5]{Reich}, whereas if $\mathcal{A}$ has a basic set of size~$1$ outside $C$, then $\mathcal{X}_{\mathcal{A}}$ is a scheme from~\cite[Theorem~6]{Reich}.

If $n$ is a prime such that $n\equiv 3 \mod~4$ and $C_1$ and $C_2$ are nontrivial orbits of the subgroup of $\aut(C)$ of index~$2$, then the partition of $G$ into the sets 
$$\{e\},~C_1,~C_2,~\{b\},~bC_1,~bC_2$$
defines a noncommutative $S$-ring $\mathcal{A}$ of rank~$6$ over $G$. The scheme $\mathcal{X}_{\mathcal{A}}$ is a noncommutative scheme of rank~$6$. Such schemes were described in~\cite[Section~7]{Reich} and earlier in a special case in~\cite{DF}.

\subsection{Isomorphisms}

In this subsection, we find the automorphism and isomorphism groups of noncommutative Tatra schemes. Due to Corollary~\ref{noncommut}, a noncommutativity of $\mathcal{X}_0$ is equivalent to
\begin{equation}\label{maineq2}
n\geq 3.
\end{equation}

Recall that $\Sigma=\aut(\mathbb{F})$ is a cyclic group of order~$d$ generated by a Frobenius automorphism $x\mapsto x^r$, $x\in \mathbb{F}$, $\GaL(2,q)=\GL(2,q)\rtimes \Sigma$, and $\PGaL(2,q)=\PGL(2,q)\rtimes \Sigma$. 

The group $\GaL(2,q)$ acts on $\Omega$ as follows: if $T\sigma\in \GaL(2,q)$, where $T\in \GL(2,q)$ and $\sigma\in \Sigma$, and $Kv\in \Omega$, then
$$(Kv)^{T\sigma}=K(Tv^\sigma),$$
where $\sigma$ acts on each component of~$v$. The kernel of this action is the group $\widetilde{K}=\{xI_2:~x\in K\}\cong K$ of order $m=|K|$. Further, if $\alpha=Kv\in \Omega$, then we denote $K(Tv^\sigma)$ by $T\alpha^\sigma$. 

The group $\Sigma$ induces the action by automorphisms on $C=\mathbb{F}^*/K$ in the following way: if $\sigma \in \Sigma$, then
$$(Kx)^\sigma=K(x^\sigma).$$ 
The kernel of the action of $\Sigma$ on $C$ is denoted by $\Sigma_0$. Observe that $\Sigma_0$ can act nontrivially on $\Omega$. For example, if $d\geq 2$ and $|K|=(q-1)/(r-1)$, then $x^r=x^{r-1}x\in Kx$ for every $x\in \mathbb{F}_q^*$ and hence $\Sigma_0=\Sigma$. On the other hand, $K(1^\sigma,x^r)=K(1,x^r)\neq K(1,x)$ for a generator $x$ of $\mathbb{F}^*$ which implies that $\Sigma$ acts nontrivially on $\Omega$. 

One can see that $K(1,x)\neq K(1^\sigma,x^\sigma)=K(1,x^\sigma)$ for a generator $x$ of $\mathbb{F}^*$ and every nontrivial $\sigma\in \Sigma$. Therefore the kernel of the action of $\Sigma$ on $\Omega$ is trivial and consequently
\begin{equation}\label{fieldker}
\Sigma^\Omega\cong \Sigma.
\end{equation}

Let 
$$\mathcal{L}=\{C\alpha:~\alpha\in \Omega\}.$$ 
Clearly, every element $\beta$ of $C\alpha$ can be uniquely presented in the form $\beta=g\alpha$, $g\in C$, and hence $|C\alpha|=|C|=n$ and $|\mathcal{L}|=|\Omega|/|C|=q+1$. One can see that the elements of $\mathcal{L}$ are in one-to-one correspondence with $1$-dimensional subspaces of $V$ and hence $\PGaL(2,q)$ induces the action on $\mathcal{L}$.

Let 
$$\GL(2,q)_K=\{T\in \GL(2,q):~\det(T)\in K\}\leq \GL(2,q).$$
Clearly, $\GL(2,q)_K$ induces the action on $\Omega$ and the kernel of this action is $\widetilde{K}$. Observe that 
$|\GL(2,q)_K|=|\SL(2,q)||K|$ and hence 
$$|\GL(2,q)_K^\Omega|=|\SL(2,q)|=|\PGL(2,q)|.$$
The kernel $N_K$ of the action of $\GL(2,q)_K$ on $\mathcal{L}$ is the group consisting of all scalar $(2\times 2)$-matrices whose determinant belongs to $K$. If $n=|\mathbb{F}^*:K|$ is odd, then $N_K=\widetilde{K}$, whereas otherwise $|N_K:\widetilde{K}|=2$. Therefore 
\begin{equation}\label{isopgl}
\GL(2,q)_K^\Omega\cong \PGL(2,q)
\end{equation}
if $n$ is odd and $\GL(2,q)_K^\Omega/N_K^\Omega\cong \PSL(2,q)$ if $n$ is even. In the latter case, $N_K^\Omega\cong N_K/\widetilde{K}\cong C_2\leq Z(\GL(2,q)_K^\Omega)$ and $|\SL(2,q)^\Omega\cap N_K^\Omega|=1$. This yields that
$$\GL(2,q)_K^\Omega\cong \PSL(2,q)\times C_2$$
if $n$ is even.

One can see that $\det(T^\sigma)=\det(T)^\sigma$ for every $T\in \GL(2,q)$ and $\sigma\in \Sigma$, where $T^\sigma$ is the matrix obtained from $T$ by the action of $\sigma$ on each entry of $T$. So every subgroup of $\Sigma$ acts, in particular, $\Sigma_0$ acts on $\GL(2,q)_K$ and the semidirect product $\GL(2,q)_K\rtimes \Sigma_0$ is well-defined.

The main result of the subsection is the theorem below.

\begin{theo}\label{mainautiso}
With the above notation, $\aut(\mathcal{X}_0)=(\GL(2,q)_K\rtimes \Sigma_0)^\Omega$ and $\iso(\mathcal{X}_0)=\GaL(2,q)^{\Omega}$.
\end{theo}

Observe that Theorem~\ref{mainautiso} does not hold if $\mathcal{X}_0$ is commutative or, equivalently, $n\leq 2$. If $n=1$, then $\mathcal{X}_0$ is trivial and hence $\aut(\mathcal{X}_0)=\sym(\Omega)$. It would be interesting to find the automorphism and isomorphism groups of $\mathcal{X}_0$ in case $n=2$. It seems that in this case the above groups can be large.

The proof of Theorem~\ref{mainautiso} will be given in the end of the section. We start with several preparatory lemmas. 

\begin{lemm}\label{neigh}
Let $g\in C$, $\Lambda_0\in \mathcal{L}$, and $\alpha_0\in \Lambda_0$. For every $\Lambda\in \mathcal{L}$ with $\Lambda\neq \Lambda_0$, there exists a unique $\alpha\in \Lambda$ such that $(\alpha_0,\alpha)\in s_g$. 
\end{lemm}
\begin{proof}
At first, let prove an existence. Let $\alpha_1\in \Lambda$. Suppose that $\langle \alpha_0,\alpha_1\rangle=h$. Then $\langle \alpha_0,h^{-1}g\alpha_1\rangle=h^{-1}g\langle \alpha_0,\alpha\rangle=h^{-1}gh=g$. So $(\alpha_0,h^{-1}g\alpha_1)\in s_g$. Clearly, $h^{-1}g\alpha_1\in \Lambda$. Thus, the required holds for $\alpha=h^{-1}g\alpha_1$.

Now let us prove a uniqueness. Assume that $(\alpha_0,\beta)\in s_g$ for some $\beta\in \Lambda$. Then $\langle \alpha_0,\beta\rangle=g$. As $\beta\in \Lambda$, we have $\beta=h^\prime \alpha_1$ for some $h^\prime\in C$. Therefore 
$$g=\langle \alpha_0,\beta\rangle=\langle \alpha_0,h^\prime \alpha_1\rangle=h^\prime\langle \alpha_0,\alpha_1\rangle=h^\prime h.$$
The above equality implies that $h^\prime=h^{-1}g$ and hence $\beta=h^\prime \alpha_1=h^{-1}g\alpha_1=\alpha$ as desired. 
\end{proof}

\begin{lemm}\label{isoscheme}
Let $f\in \sym(\Omega)$ such that $r_C^f=r_C$. Then the following statements hold:
\begin{enumerate}

\tm{1} $\Lambda^f\in \mathcal{L}$ for every $\Lambda\in \mathcal{L}$ and hence $f$ induces the permutation $f^{\mathcal{L}}\in \sym(\mathcal{L})$;

\tm{2} for every $\Lambda\in \mathcal{L}$ and every $\alpha\in \Lambda$, the mapping $f_{\Lambda,\alpha}$ from $C$ to itself which maps $g\in C$ to $g^\prime\in C$ such that $(g\alpha)^f=g^\prime\alpha^f$ is a well-defined permutation on~$C$.
\end{enumerate}
\end{lemm}

\begin{proof}
To prove Statement~$(1)$, it is enough to verify that if $\alpha,\beta\in \Omega$ are such that $\beta=g\alpha$ for some $g\in C$, then $\beta^f=g^\prime\alpha^f$ for some $g^\prime\in C$. Since $\beta=g\alpha$, we have $\langle \alpha,\beta \rangle=0$. So $(\alpha,\beta)\in r_C$ by Eq.~\eqref{rc}. The assumption of the lemma implies that $(\alpha^f,\beta^f)\in r_C$ and hence $\langle \alpha^f,\beta^f \rangle=0$ by Eq.~\eqref{rc}. Therefore $\beta^f=g^\prime\alpha^f$ for some $g^\prime\in C$ as desired.

Let $\Lambda\in \mathcal{L}$ and $\alpha\in \Lambda$. Statement~$(1)$ implies that
$$\{(g\alpha)^f:~g\in C\}=\Lambda^f=\{g^\prime\alpha^f:~g^\prime\in C\}.$$
Since every element of $\Lambda^f$ can be uniquely presented in the form $(g\alpha)^f$, $g\in C$, as well as in the form $g^\prime\alpha^f$, $g^\prime\in C$, the mapping $f_{\Lambda,\alpha}$ defined in Statement~$(2)$ of the lemma is a permutation on~$C$.
\end{proof}

\begin{corl}\label{actonl}
With the above notation,  $\mathcal{L}$ is an imprimitivity system for $\aut(\mathcal{X}_0)$.
\end{corl}

\begin{proof}
Since $r_C$ is a union of some basic relations of $\mathcal{X}_0$, we have $r_C^f=r_C$ for every $f\in \aut(\mathcal{X}_0)$. So $\mathcal{L}$ is an imprimitivity system for $\aut(\mathcal{X}_0)$ by Statement~$(1)$ of Lemma~\ref{isoscheme}.
\end{proof}

As $n_{r_g}=1$ for every $g\in C$ (Proposition~\ref{scheme}(4)), $r_g$ is a bijection on $\Omega$ which maps $\alpha\in \Omega$ to $g\alpha\in \Omega$. Clearly, $r_g\in Z(\GL(2,q))^\Omega$.

\begin{lemm}\label{autker}
Let $N_0$ be the kernel of the action of $\aut(\mathcal{X}_0)$ on $\mathcal{L}$. Then $N_0$ is trivial if $n$ is odd and $N_0=\langle r_h \rangle$, where $h\in C$ with $|h|=2$, if $n$ is even. In particular, $|N_0|\leq 2$.
\end{lemm}
 
\begin{proof}
Let $\Lambda\in \mathcal{L}$ and $g_0$ a generator of $C$. Clearly, $r_{g_0}$ induces a permutation $r_{g_0}^\Lambda$ on $\Lambda$ which maps every $\alpha\in \Lambda$ to $g_0\alpha\in \Lambda$. Since $g_0$ is a generator of $C$, the group $\langle r_{g_0}^\Lambda \rangle\cong C$ is regular on $\Lambda$. This implies that $C_{\sym(\Lambda)}(\langle r_{g_0}^\Lambda \rangle)=\langle r_{g_0}^\Lambda \rangle$. One can see that $f^{-1}r_{g_0}f=r_{g_0}^f=r_{g_0}$ for every $f\in N_0\leq \aut(\mathcal{X}_0)$. So
\begin{equation}\label{centralizer}
N_0^\Lambda\leq C_{\sym(\Lambda)}(\langle r_{g_0}^\Lambda \rangle)=\langle r_{g_0}^\Lambda \rangle 
\end{equation}
for every $\Lambda\in \mathcal{L}$. 

Suppose that $f\in N_0$ acts trivially on some $\Lambda_0\in \mathcal{L}$. From Lemma~\ref{neigh} it follows that for every $\alpha\in \Omega\setminus \Lambda_0$ there exists $\alpha_0\in \Lambda_0$ and $g\in C$ such that $\{\alpha\}=\alpha_0s_g \cap \Lambda$, where $\Lambda\in \mathcal{L}$ contains $\alpha$. Therefore 
$$\{\alpha^f\}=\alpha_0^fs_g^f \cap \Lambda^f=\alpha_0s_g \cap \Lambda=\{\alpha\}.$$
This yields that $f$ is trivial and consequently the action of $N_0$ on $\Lambda_0$ is faithful. Thus, $N_0^\Lambda\cong N_0^{\Lambda_0}$ for every $\Lambda\in \mathcal{L}$. Together with Eq.~\eqref{centralizer}, this implies that $N_0$ is cyclic and a generator $f_0$ of $N_0$ can be presented in the form
$$f_0=\prod \limits_{\Lambda\in \mathcal{L}} (r_{g_0}^\Lambda)^{j_\Lambda}$$
for some $j_\Lambda\in \mathbb{Z}_n$.

Let $\Lambda,\Lambda^\prime\in \mathcal{L}$ and $g\in C$. Then $s=s_g\cap (\Lambda\times \Lambda^\prime)\neq\varnothing$ and $s$ is a matching between $\Lambda$ and $\Lambda^\prime$ by Lemma~\ref{neigh}. Since $s_g^{f_0}=s_g$, we conclude that $(r_{g_0}^\Lambda)^{j_\Lambda}s(r_{g_0}^{\Lambda^\prime})^{-j_{\Lambda^\prime}}=s$ and hence
\begin{equation}\label{autker1}
s^{-1}(r_{g_0}^\Lambda)^{j_\Lambda}s=(r_{g_0}^{\Lambda^\prime})^{j_{\Lambda^\prime}}.
\end{equation}
On the other hand, $r_{g_0}s_gr_{g_0}=s_g$ by Proposition~\ref{scheme} which implies that $(r_{g_0}^\Lambda)s(r_{g_0}^{\Lambda^\prime})=s$ and consequently
\begin{equation}\label{autker2}
s^{-1}(r_{g_0}^\Lambda)s=(r_{g_0}^{\Lambda^\prime})^{-1}.
\end{equation}
From Eqs.~\eqref{autker1} and~\eqref{autker2} it follows that $j_{\Lambda^\prime}=-j_\Lambda$ for all $\Lambda,\Lambda^\prime\in \mathcal{L}$. This yields that $j_\Lambda=0$ for every $\Lambda\in \mathcal{L}$ or $n$ is even and $j_\Lambda=n/2$ for every $\Lambda\in \mathcal{L}$. In the former case, $f_0$ and hence $N_0$ are trivial. In the latter one, $n$ is even and $f_0=r_h$ for $h\in C$ with $|h|=2$. To finish the proof, it remains to verify that $r_h\in N_0$ in the latter case. Clearly, $r_h$ fixes every $\Lambda\in \mathcal{L}$ as a set. From Statements~$(1)$ and~$(2)$ of Proposition~\ref{scheme} it follows that $r_g^{r_h}=r_g$ and $s_g^{r_h}=s_{gh^2}=s_g$, respectively, for every $g\in C$. Thus, $r_h\in N_0$ as desired.
\end{proof}

\begin{lemm}\label{actiongl}
Let $f=(T\sigma)^\Omega\in \GaL(2,q)^\Omega$, where $T\in \GL(2,q)$ and $\sigma\in \Sigma$. Then $r_g^f=r_{g^\sigma}$ and $s_g^f=s_{\det(T)g^\sigma}$ for every $g\in C$.
\end{lemm}

\begin{proof}
Let $g\in C$. Suppose that $(\alpha_1,\alpha_2)\in r_g$. Then $\alpha_2=g\alpha_1$. So 
$$\alpha_2^f=(g\alpha_1)^f=Tg^\sigma\alpha_1^\sigma=g^\sigma T\alpha_1^\sigma=g^\sigma (\alpha_1)^f.$$
Therefore $(\alpha_1^f,\alpha_2^f)\in r_{g^\sigma}$ and consequently $r_g^f=r_{g^\sigma}$.

Now suppose that $(\alpha_1,\alpha_2)\in s_g$. Then $\langle \alpha_1,\alpha_2\rangle=g$. This implies that
$$\langle \alpha_1^f,\alpha_2^f \rangle=\langle T\alpha_1^\sigma,T\alpha_2^\sigma \rangle=K\det(T\alpha_1^\sigma,T\alpha_2^\sigma)=K(\det(T)\det(\alpha_1,\alpha_2)^\sigma)=\det(T)g^\sigma$$
Therefore $(\alpha_1^f,\alpha_2^f)\in s_{\det(T)g^\sigma}$ and hence $s_g^f=s_{\det(T)g^\sigma}$.
\end{proof}

The corollary below immediately follows from Lemma~\ref{actiongl}.

\begin{corl}\label{oneside}
With the notation of Theorem~\ref{mainautiso}, $\aut(\mathcal{X}_0)\geq (\GL(2,q)_K\rtimes \Sigma_0)^\Omega$.
\end{corl}

\begin{corl}\label{isodrg}
The graphs $(\Omega,s_h)$ and $(\Omega,s_g)$ are isomorphic for all~$h,g\in C$.
\end{corl}

\begin{proof}
To prove the corollary, it is enough to show that $(\Omega,s_e)$ and $(\Omega,s_g)$ are isomorphic for every~$g\in C$. Let $T\in \GL(2,q)$ such that $\det(T)\in g$ and $f_T\in \sym(\Omega)$ the permutation $\alpha\mapsto T\alpha$, $\alpha \in \Omega$, induced by $T$ on $\Omega$. Then $s_e^{f_T}=s_{K\det(T)}=s_g$ by Lemma~\ref{actiongl} and hence $f_T$ is an isomorphism from $(\Omega,s_e)$ to $(\Omega,s_g)$.
\end{proof}

The next statement follows from the classification of $2$-transitive permutation groups (see~\cite[Theorem~5.1]{Cam1}). It can also be found, e.g., in~\cite{Cam2}.

\begin{lemm}\label{overgroup}
Let $q>23$ and $G\leq S_{q+1}$ such that $G\geq \PSL(2,q)$. Then $G\leq \PGaL(2,q)$ or $G$ is isomorphic to one of the groups $A_{q+1}$, $S_{q+1}$.
\end{lemm}

\begin{lemm}\label{notas}
If $q\geq 5$, then $\aut(\mathcal{X}_0)^{\mathcal{L}}$ is not isomorphic to $A_{q+1}$ or $S_{q+1}$.
\end{lemm}

\begin{proof}
Assume the contrary that $\aut(\mathcal{X}_0)^{\mathcal{L}}$ is isomorphic to $A_{q+1}$ or $S_{q+1}$. Corollary~\ref{oneside} implies that $\aut(\mathcal{X}_0)\geq \SL(2,q)$ and hence $\aut(\mathcal{X}_0)$ is transitive on $\Omega$. So 
\begin{equation}\label{ind1}
|\aut(\mathcal{X}_0):\aut(\mathcal{X}_0)_{\alpha}|=|\Omega|=(q+1)n
\end{equation}
for every $\alpha\in \Omega$.

Let $\Lambda\in \mathcal{L}$ and $\alpha\in \Lambda$. Due to Corollary~\ref{actonl}, $\Lambda$ is a block of $\aut(\mathcal{X}_0)$. Clearly, $\aut(\mathcal{X}_0)_{\alpha}\leq \aut(\mathcal{X}_0)_{\{\Lambda\}}$. Since $\Lambda$ is a class of the thin radical parabolic $r_C$, we conclude that 
\begin{equation}\label{thinrad2}
\aut(\mathcal{X}_0)_{\{\Lambda\}}^\Lambda~\text{is a regular group isomorphic to}~C
\end{equation}
by Eqs.~\eqref{thinrad0} and~\eqref{thinrad} and hence $\aut(\mathcal{X}_0)^\Lambda_{\alpha}$ is trivial. This yields that $\aut(\mathcal{X}_0)_{\alpha}$ is the kernel of the action of $\aut(\mathcal{X}_0)_{\{\Lambda\}}$ on $\Lambda$. Therefore
\begin{equation}\label{ind2}
\aut(\mathcal{X}_0)_{\alpha}\trianglelefteq \aut(\mathcal{X}_0)_{\{\Lambda\}}~\text{and}~|\aut(\mathcal{X}_0)_{\{\Lambda\}}:\aut(\mathcal{X}_0)_{\alpha}|=|\Lambda|=n.
\end{equation}

One can see that $r_g$, $g\in C$, is a semiregular permutation on $\Omega$ fixing every element of $\mathcal{L}$ as a set. Together with Lemma~\ref{autker}, this implies that 
\begin{equation}\label{stabpoint}
\aut(\mathcal{X}_0)_{\alpha}^\mathcal{L}\cong\aut(\mathcal{X}_0)_{\alpha}
\end{equation}
and $|\aut(\mathcal{X}_0)_{\{\Lambda\}}^\mathcal{L}|=|\aut(\mathcal{X}_0)_{\{\Lambda\}}|/|N_0|$. Therefore
$$|\aut(\mathcal{X}_0)^\mathcal{L}|/|\aut(\mathcal{X}_0)_{\alpha}^\mathcal{L}|=|\aut(\mathcal{X}_0)|/(|N_0||\aut(\mathcal{X}_0)_{\alpha}|)=(q+1)n/|N_0|$$
by Eq.~\eqref{ind1} and
$$|\aut(\mathcal{X}_0)_{\{\Lambda\}}^\mathcal{L}|/|\aut(\mathcal{X}_0)_{\alpha}^\mathcal{L}|=|\aut(\mathcal{X}_0)_{\{\Lambda\}}|/(|N_0||\aut(\mathcal{X}_0)_{\alpha}|)=n/|N_0|$$
by Eq.~\eqref{ind2}. From the above two equalities it follows that
$$|\aut(\mathcal{X}_0)^\mathcal{L}|/|\aut(\mathcal{X}_0)_{\{\Lambda\}}^\mathcal{L}|=q+1.$$
Thus, $\aut(\mathcal{X}_0)_{\{\Lambda\}}^\mathcal{L}$ is a subgroup of $\aut(\mathcal{X}_0)^\mathcal{L}$ of index~$q+1$ having a normal subgroup $\aut(\mathcal{X}_0)_{\alpha}^\mathcal{L}$ of index~$n/|N_0|$. In view of Eq.~\eqref{maineq2} and Lemma~\ref{autker}, we have $n/|N_0|\geq 2$.

Since $\aut(\mathcal{X}_0)^{\mathcal{L}}$ is isomorphic to $A_{q+1}$ or $S_{q+1}$ by the assumption and $q\geq 5$, we conclude that $\aut(\mathcal{X}_0)_{\{\Lambda\}}^\mathcal{L}$ is isomorphic to $A_q$ or $S_q$. Due to $q\geq 5$, the group $A_q$ is simple and the group $S_q$ has a unique proper nontrivial normal subgroup $A_q$. So $\aut(\mathcal{X}_0)_{\alpha}^\mathcal{L}$ being a proper normal subgroup of $\aut(\mathcal{X}_0)_{\{\Lambda\}}^\mathcal{L}$ is trivial or isomorphic to $A_q$.

In the first case, $\aut(\mathcal{X}_0)_{\alpha}$ is also trivial by Eq.~\eqref{stabpoint}. However, $\aut(\mathcal{X}_0)_{\alpha}\geq \SL(2,q)^\Omega_{\alpha}$ by Corollary~\ref{oneside} and 
$$|\SL(2,q)^\Omega_{\alpha}|=\frac{|\SL(2,q)|}{|K||\Omega|}=\frac{(q^2-1)(q^2-q)}{(q^2-1)(q-1)}=q>1,$$
a contradiction.

In the second one, $\aut(\mathcal{X}_0)_{\{\Lambda\}}^\mathcal{L}\cong S_q$ and 
$$n/|N_0|=|\aut(\mathcal{X}_0)_{\{\Lambda\}}^\mathcal{L}|/|\aut(\mathcal{X}_0)_{\alpha}^\mathcal{L}|=2.$$
The latter equality, Eq.~\eqref{maineq2}, and Lemma~\ref{autker} yield that $n=4$ and $N_0=\langle r_h \rangle\cong C_2$, where $h\in C$ with $|h|=2$. Clearly, $r_h^f=r_h$ for every $f\in \aut(\mathcal{X}_0)$ and hence $r_h\in Z(\aut(\mathcal{X}_0)_{\{\Lambda\}})$. Therefore $\aut(\mathcal{X}_0)_{\{\Lambda\}}$ is a double cover of $S_q$. As $\aut(\mathcal{X}_0)_{\alpha}^\mathcal{L}\cong\aut(\mathcal{X}_0)_{\alpha}$, we conclude that $\aut(\mathcal{X}_0)_{\alpha}\cong A_q$. A nonsplit double cover of $S_q$ does not have a normal subgroup isomorphic to $A_q$ and hence $\aut(\mathcal{X}_0)_{\alpha}$ can not be normal in $\aut(\mathcal{X}_0)_{\{\Lambda\}}$, a contradiction to the first part of Eq.~\eqref{ind2}. So $\aut(\mathcal{X}_0)_{\{\Lambda\}}\cong S_q\times C_2$. Then 
$$\aut(\mathcal{X}_0)_{\{\Lambda\}}/\aut(\mathcal{X}_0)_{\alpha}\cong C_2\times C_2.$$
On the other hand, 
$$\aut(\mathcal{X}_0)_{\{\Lambda\}}/\aut(\mathcal{X}_0)_{\alpha}\cong \aut(\mathcal{X}_0)_{\{\Lambda\}}^\Lambda\cong C\cong C_4$$
by Eq.~\eqref{thinrad2}, a contradiction.      
\end{proof}

\begin{prop}\label{auttatra}
With the notation of Theorem~\ref{mainautiso}, $\aut(\mathcal{X}_0)=(\GL(2,q)_K\rtimes \Sigma_0)^\Omega$.
\end{prop}

\begin{proof}
The inclusion $\aut(\mathcal{X}_0)\geq (\GL(2,q)_K\rtimes \Sigma_0)^\Omega$ follows from Lemma~\ref{actiongl}. Let us prove the reverse inclusion $\aut(\mathcal{X}_0)\leq (\GL(2,q)_K\rtimes \Sigma_0)^{\Omega}$. It is easy to see that $\GL(2,q)_K\geq \SL(2,q)$ and hence 
\begin{equation}\label{pgl} 
\GL(2,q)_K^{\mathcal{L}}\geq \PSL(2,q).
\end{equation}

Since $\aut(\mathcal{X}_0)\geq \GL(2,q)_K^\Omega$, Eq.~\eqref{pgl} implies that $\aut(\mathcal{X}_0)^\mathcal{L}\geq \PSL(2,q)$. The group $\PSL(2,q)$ (acting on $\mathcal{L}$) is $2$-transitive by~\cite[Theorem~5.1]{Cam1} and hence $\aut(\mathcal{X}_0)^\mathcal{L}$ so is. If $q\leq 23$, then one can verify by the computer calculation using~\cite{GAP} and the list of $2$-transitive groups from~\cite[Theorem~5.1]{Cam1} that
\begin{equation}\label{autsubgroup}
\aut(\mathcal{X}_0)^{\mathcal{L}}\leq \PGaL(2,q).
\end{equation}  
Otherwise, Eq.~\eqref{autsubgroup} holds or $\aut(\mathcal{X}_0)^{\mathcal{L}}$ is isomorphic to one of the groups $A_{q+1}$, $S_{q+1}$ by Lemma~\ref{overgroup}. However, the latter is impossible by Lemma~\ref{notas}. Thus, Eq.~\eqref{autsubgroup} holds in any case. Together with Lemma~\ref{autker}, this implies that $\aut(\mathcal{X}_0)\leq \GaL(2,q)^\Omega$.

Assume that there exists $f\in \aut(\mathcal{X}_0)\cap (\GaL(2,q)^\Omega\setminus (\GL(2,q)_K\rtimes \Sigma_0)^\Omega)$. Then $f=(T\sigma)^\Omega$, where $T\in \GL(2,q)$, $\sigma \in \Sigma$, and one of the following conditions holds: $(1)$ $\det(T)\notin K$; $(2)$ $\sigma$ acts nontrivially on $C$. If the first condition holds, then $s_e^f=s_{\det(T)}\neq s_e$ by Lemma~\ref{actiongl}, a contradiction to $f\in \aut(\mathcal{X}_0)$. If the second one holds, then there exist $g\in C$ such that $g^\sigma\neq g$. So $r_g^f=r_{g^\sigma}\neq r_g$, a contradiction to $f\in \aut(\mathcal{X}_0)$. Therefore $\aut(\mathcal{X}_0)\leq (\GL(2,q)_K\rtimes \Sigma_0)^{\Omega}$ and we are done. 
\end{proof}

\begin{corl}\label{x0schur}
The scheme $\mathcal{X}_0$ is schurian.
\end{corl}
  
\begin{proof}
Due to Proposition~\ref{auttatra}, we have $\aut(\mathcal{X}_0)\geq \GL(2,q)_K^\Omega$. Let $g\in C$. Since $\SL(2,q)\leq \GL(2,q)_K$ is transitive on $V\setminus \{0\}$, $\GL(2,q)_K$ is transitive on $r_g$. The group $\GL(2,q)$ is transitive on the set of all ordered bases of $V$. So for all $(\alpha_1,\alpha_2),(\beta_1,\beta_2)\in s_g$, there exists $T\in \GL(2,q)$ such that $(T\alpha_1,T\alpha_2)=(\beta_1,\beta_2)$. As $g=\langle \beta_1,\beta_2\rangle=\langle T\alpha_1,T\alpha_2\rangle=\det(T)\langle \alpha_1,\alpha_2\rangle=\det(T)g$, we conclude that $\det(T)\in K$ and hence $T\in \GL(2,q)_K$. Therefore $\GL(2,q)_K$ is transitive on every basic relation of $\mathcal{X}_0$. Thus, $\mathcal{X}_0$ is schurian. 
\end{proof}

\begin{prop}\label{isotatra}
With the notation of Theorem~\ref{mainautiso}, $\iso(\mathcal{X}_0)=\GaL(2,q)^{\Omega}$.
\end{prop}

\begin{proof}
The inclusion $\iso(\mathcal{X}_0)\geq \GaL(2,q)^{\Omega}$ follows from Lemma~\ref{actiongl}. Let us prove the reverse inclusion $\iso(\mathcal{X}_0)\leq \GaL(2,q)^{\Omega}$. If $f\in\iso(\mathcal{X}_0)$, then $f$ maps any basic relation of $\mathcal{X}_0$ of valency~$1$ to a basic relation of valency~$1$ and hence preserves the thin radical parabolic $r_C$ of $\mathcal{X}_0$. So $r_C^f=r_C$ for every $f\in \iso(\mathcal{X}_0)$. Therefore $\iso(\mathcal{X}_0)$ acts on $\mathcal{L}$ by Lemma~\ref{isoscheme}(1). One can see that
$$\iso(\mathcal{X}_0)^{\mathcal{L}}\trianglerighteq\aut(\mathcal{X}_0)^{\mathcal{L}}\geq \PSL(2,q),$$
where the last inequality holds by Proposition~\ref{auttatra}. If $q\leq 23$, then one can verify by the computer calculation using~\cite{GAP} and the list of $2$-transitive groups from~\cite[Theorem~5.1]{Cam1} that
\begin{equation}\label{isosubgroup}
\iso(\mathcal{X}_0)^{\mathcal{L}}\leq \PGaL(2,q).
\end{equation} 
Otherwise, the Eq.~\eqref{isosubgroup} holds or $\iso(\mathcal{X}_0)^{\mathcal{L}}$ is isomorphic to one of the groups $A_{q+1}$, $S_{q+1}$ by Lemma~\ref{overgroup}. However, the latter is impossible because $\iso(\mathcal{X}_0)^\mathcal{L}$ has a normal subgroup $\aut(\mathcal{X}_0)^\mathcal{L}$ which is contained in $\PGaL(2,q)$ by Proposition~\ref{auttatra}. Therefore Eq.~\eqref{isosubgroup} holds in any case. 

Clearly, $\GaL(2,q)^\mathcal{L}=\PGaL(2,q)$. As $\iso(\mathcal{X}_0)\geq \GaL(2,q)^{\Omega}$, we conclude that $\iso(\mathcal{X}_0)^\mathcal{L}\geq \GaL(2,q)^\mathcal{L}=\PGaL(2,q)$. Together with Eq.~\eqref{isosubgroup}, this yields that  $\iso(\mathcal{X}_0)^{\mathcal{L}}=\PGaL(2,q)$. Therefore 
\begin{equation}\label{inker}
\iso(\mathcal{X}_0)\leq \GaL(2,q)^{\Omega}G_0,
\end{equation}  
where $G_0$ is the kernel of the action of $\iso(\mathcal{X}_0)$ on $\mathcal{L}$.

Let us show that $G_0$ is semiregular. To do this, it is enough to show that $G_0$ lies in a centralizer of some transitive group. If $n$ is odd, then $|G_0\cap \aut(\mathcal{X}_0)|=1$ by Lemma~\ref{autker} which implies $G_0\leq C_{\sym(\Omega)}(\aut(\mathcal{X}_0))$. We are done because $\aut(\mathcal{X}_0)$ is transitive. If $n$ is even, then $G_0\cap \aut(\mathcal{X}_0)=\langle r_h \rangle$, where $h\in C$ with $|h|=2$, by Lemma~\ref{autker}. So $f_1^{-1}f_2^{-1}f_1f_2\in\langle r_h \rangle$ for all $f_1\in \aut(\mathcal{X}_0)$ and $f_2\in G_0$. Note that $r_h^{f_1}=r_h$ and hence $f_1$ and $r_h$ commute. Together with $r_h^2=r_e=1_{\Omega}$, this implies that
$$f_2^{-1}f_1^2f_2=f_1^2.$$
So $G_0\leq C_{\sym(\Omega)}(F)$, where $F=\langle f^2:~f\in \aut(\mathcal{X}_0)\rangle$. In view of Proposition~\ref{auttatra}, we have $\aut(\mathcal{X}_0)\geq \SL(2,q)$. As $n$ is even, $q$ must be odd. Therefore every transvection $T_{ij}(x)$ is a square of $T_{ij}(x/2)$ and consequently $F\geq \langle T_{ij}(x):~x\in \mathbb{F} \rangle=\SL(2,q)^\Omega$. Thus, $G_0\leq C_{\sym(\Omega)}(\SL(2,q)^\Omega)$ and we prove the desired inclusion for the transitive group $\SL(2,q)^\Omega$.

By the above paragraph, $G_0$ is semiregular. Therefore $|G_0|\leq |\Omega|/|\mathcal{L}|=n$. On the other hand, $G_0\geq Z(\GL(2,q))^\Omega\cong C_n$ because $\iso(\mathcal{X}_0)\geq \GaL(2,q)^{\Omega}$. We conclude that $G_0=Z(\GL(2,q))^\Omega$. In particular, $G_0\leq  \GaL(2,q)^{\Omega}$. Together with Eq.~\eqref{inker}, this yields that $\iso(\mathcal{X}_0)\leq \GaL(2,q)^{\Omega}$. Thus, $\iso(\mathcal{X}_0)=\GaL(2,q)^{\Omega}$ as desired.
\end{proof}

Theorem~\ref{mainautiso} is an immediate consequence of Propositions~\ref{auttatra} and~\ref{isotatra}.

\section{Directed strongly regular graphs}

In this section, we keep the notation from the previous one. Suppose that $n=p$ is a prime such that $p \equiv 3 \mod~4$. Then Conditions~\eqref{maineq} and~\eqref{maineq2} necessary holds and hence one can construct the scheme $\mathcal{X}_0$. 

Let $M$ be a subgroup of $\aut(C)\cong C_{p-1}$ of index~$2$ and $\mathcal{A}=\cyc(K,C)$. As $|\aut(C):M|=2$ and $\aut(C)$ is regular on $C^\#$, the $S$-ring $\mathcal{A}$ has two basic sets distinct from $\{e\}$, say, $C_1$ and $C_2$. Since $p \equiv 3 \mod~4$, we conclude that $|M|=(p-1)/2$ is odd and hence $M$ does not contain the automorphism which inverses every element of~$C$. Therefore $C_1^{(-1)}=C_2$. The Cayley digraphs $\cay(C,C_1)$ and $\cay(C,C_2)$ are Paley tournaments (see~\cite{BCN,BM}). The following well-known lemma immediately follows from the results on cyclotomic numbers modulo~$2$~\cite[Eq.~(19)]{Dickson} (see also~\cite[Theorem~7]{Reich}).

\begin{lemm}\label{paley}
With the above notation,
$$\underline{C_i}^2=\frac{p-3}{4}\underline{C_i}+\frac{p+1}{4}\underline{C_{3-i}},~i\in\{1,2\}.$$
\end{lemm} 

The $S$-ring $\mathcal{A}$ is normal by~\cite[Theorem~4.1]{EP}, i.e. $C_\r$ is normal in $\aut(\mathcal{A})$ or, equivalently, $\aut(\mathcal{A})_e\leq \aut(C)$. Clearly, $\aut(\mathcal{A})_e\geq M$. Since $C$ is cyclic, we obtain $\aut(\mathcal{A})_e=M$ and hence
\begin{equation}\label{autpaley}
\aut(\mathcal{A})=C_\r\rtimes M.
\end{equation}


Let $i\in\{1,2\}$, $g\in C$, and $\Gamma=\Gamma(i,g)$ the digraph with vertex set $\Omega$ and arc set $s_g\cup r_{C_i}$.

\begin{prop}\label{dsrg}
The digraph $\Gamma=\Gamma(i,g)=(\Omega,s_g\cup r_{C_i})$ is a directed strongly regular graph if and only if $q-1=p(p-3)/4$. If the latter is the case, then $\Gamma$ has parameters
$$\left((q+1)p,q+\frac{p-1}{2},q,\frac{q-1}{p},\frac{q-1}{p}+1\right).$$
\end{prop}

\begin{proof}
Let $A=A(\Gamma)$ be the adjacency matrix of $\Gamma$. To prove the proposition, it is enough to verify that $A$ satisfies the equalities from Statements~$(1)$ and~$(2)$ of Lemma~\ref{dsrg0}. Note that $s_g\cup r_{C_i}$ is a regular binary relation because it is union of basic relations of the scheme $\mathcal{X}$. Moreover,
$$n_{s_g\cup r_{C_i}}=n_{s_g}+n_{r_{C_i}}=n_{s_g}+\sum \limits_{h\in C_i} n_{r_h}=q+\frac{p-1}{2}$$
be Proposition~\ref{scheme}(4). Therefore
$$AJ=JA=\left(q+\frac{p-1}{2}\right)J$$
and consequently the equality from Statement~$(1)$ of Lemma~\ref{dsrg0} holds for $A$.

Let us compute $A^2$. Clearly, $A=A(s_g)+A(r_{C_i})=B_g+A_{C_i}$. So
\begin{equation}\label{eq0}
A^2=(B_g+A_{C_i})^2=B_g^2+B_gA_{C_i}+A_{C_i}B_g+A_{C_i}^2.
\end{equation}
Proposition~\ref{scheme}(3) implies that
\begin{equation}\label{eq1}
B_g^2=qA_{1}+m(J-A_C)=qA_e+\frac{q-1}{p}\left(\sum\limits_{h\in C} B_h\right).
\end{equation}
One can see that
\begin{equation}\label{eq2}
B_gA_{C_i}+A_{C_i}B_g=B_g(A_{C_i}+A_{C_{3-i}})=B_g\left(\sum \limits_{h\in C^\#} A_h\right)=\sum \limits_{h\in C^\#}B_{gh}=\sum \limits_{h\in C\setminus\{g\}}B_{h},
\end{equation}
where the first equality follows from Lemma~\ref{aux}(2) and the equality $C_i^{(-1)}=C_{3-i}$, the second one from the equality $C_i\cup C_i^{(-1)}=C^\#$, and the third one from Proposition~\ref{scheme}(2). Finally,
\begin{equation}\label{eq3}
A_{C_i}^2=A_{\underline{C_i}^2}=A_{\frac{p-3}{4}\underline{C_i}+\frac{p+1}{4}\underline{C_{3-i}}}=\frac{p-3}{4}A_{C_i}+\frac{p+1}{4}A_{C_{3-i}},
\end{equation}
where the the first equality follows from Lemma~\ref{aux}(1), whereas the second one from Lemma~\ref{paley}. Substituting the expressions for $B_g^2$, $B_gA_{C_i}+A_{C_i}B_g$, and $A_{C_i}^2$ from Eqs.~\eqref{eq1},~\eqref{eq2}, and~\eqref{eq3}, respectively, to Eq.~\eqref{eq0}, we obtain
$$A^2=qA_e+\frac{q-1}{p}\left(\sum\limits_{h\in C} B_h\right)+\sum \limits_{h\in C\setminus\{g\}}B_{h}+\frac{p-3}{4}A_{C_i}+\frac{p+1}{4}A_{C_{3-i}}=$$
$$=qA_e+\frac{q-1}{p}B_g+\frac{p-3}{4}A_{C_i}+(\frac{q-1}{p}+1)\left(\sum \limits_{h\in C\setminus\{g\}}B_{h}\right)+\frac{p+1}{4}A_{C_{3-i}}.$$
The above equality implies that the equality from Statement~$(2)$ of Lemma~\ref{dsrg0} holds for $A$ if and only if $\frac{q-1}{p}=\frac{p-3}{4}$ or, equivalently, $q-1=p(p-3)/4$, and if the latter is the case, then $\Gamma$ has the required in the proposition parameters.
\end{proof}

\begin{prop}\label{isodsrg}
Given $i\in\{1,2\}$, the digraphs $\Gamma(i,h)$ and $\Gamma(i,g)$ are isomorphic for all~$h,g\in C$.
\end{prop}

\begin{proof}
To prove the lemma, it is enough to show that $\Gamma(i,e)$ and $\Gamma(i,g)$ are isomorphic for every $g\in C^\#$. Let $T\in \GL(2,q)$ such that $\det(T)\in g$ and $f_T\in \sym(\Omega)$ the permutation $\alpha\mapsto T\alpha$, $\alpha \in \Omega$, induced by $T$ on $\Omega$. Then $r_{C_i}^{f_T}=r_{C_i}$ and $s_e^{f_T}=s_{K\det(T)}=s_g$ by Lemma~\ref{actiongl} and hence $f_T$ is an isomorphism from $\Gamma(i,e)$ to $\Gamma(i,g)$.
\end{proof}

Further, we are going to prove that if $\Gamma(1,h)$ and $\Gamma(2,g)$, $h,g\in C$, are DSRGs, then they are nonisomorphic. The key step towards this goal is the following statement.

\begin{prop}\label{isonorm}
Let $f\in \sym(\Omega)$ such that $s_g^f=s_g$ for some $g\in C$ and $\{r_{C_1}^f,r_{C_2}^f\}=\{r_{C_1},r_{C_2}\}$. Then $f\in \iso(\mathcal{X}_0)$.
\end{prop}

\begin{proof}
We may assume that $g=e$. Indeed, otherwise let $T\in \GL(2,q)$ such that $\det(T)\in g$ and $f_T$ the permutation induced by $T$ on $\Omega$. Then $s_e^{f_Tff_T^{-1}}=s_g^{ff_T^{-1}}=s_g^{f_T^{-1}}=s_e$, where the first and third equalities hold by Lemma~\ref{actiongl}. Due to Proposition~\ref{isotatra}, we have $f_T\in \iso(\mathcal{X}_0)$ and hence $f\in \iso(\mathcal{X}_0)$ if and only if $f_Tff_T^{-1}\in \iso(\mathcal{X}_0)$.

The condition $\{r_{C_1}^f,r_{C_2}^f\}=\{r_{C_1},r_{C_2}\}$ implies that $r_C^f=r_C$. So Lemma~\ref{isoscheme} holds for $f$, i.e. $f$ induces a permutation on $\mathcal{L}$ and for all $\Lambda\in \mathcal{L}$ and $\alpha\in \Lambda$, the mapping $f_{\Lambda,\alpha}:g\rightarrow g^{\prime}$ such that $(g\alpha)^f=g^{\prime}\alpha^f$ is a permutation on $C$. Further, we prove three lemmas on properties of $f_{\Lambda,\alpha}$'s required for the proof of the proposition.

\begin{lemm}\label{iso1}
With the above notation, $f_{\Lambda,\alpha}\in \aut(C)$ for all $\Lambda\in \mathcal{L}$ and $\alpha\in \Lambda$.
\end{lemm}

\begin{proof}
Let $\Lambda\in \mathcal{L}$ and $\alpha\in \Lambda$. Since $\{r_{C_1}^f,r_{C_2}^f\}=\{r_{C_1},r_{C_2}\}$, for each $i\in \{1,2\}$ there exists $j\in \{1,2\}$ such that $r_{C_i}^f=r_{C_j}$. Therefore
$$gh^{-1}\in C_i\Leftrightarrow (h\alpha,g\alpha)\in r_{C_i}\Leftrightarrow ((h\alpha)^f,(g\alpha)^f)\in r_{C_i}^f\Leftrightarrow $$
$$\Leftrightarrow (h^{f_{\Lambda,\alpha}}\alpha^f,g^{f_{\Lambda,\alpha}}\alpha^f)\in r_{C_j}\Leftrightarrow g^{f_{\Lambda,\alpha}}(h^{f_{\Lambda,\alpha}})^{-1}\in C_j$$
for all $h,g\in C$. So $\{\cay(G,C_1)^{f_{\Lambda,\alpha}},\cay(G,C_2)^{f_{\Lambda,\alpha}}\}=\{\cay(G,C_1),\cay(G,C_2)\}$. Together with $C_2=C_1^{(-1)}$, this implies that $f_{\Lambda,\alpha}\in \aut(\mathcal{A})$ or $f_{\Lambda,\alpha}\sigma_0\in \aut(\mathcal{A})$, where $\sigma_0\in \aut(C)$ such that $g^{\sigma_0}=g^{-1}$ for every $g\in C$. In both cases, $f_{\Lambda,\alpha}\in \Hol(C)$ by Eq.~\eqref{autpaley}. Since $\alpha^f=e^{f_{\Lambda,\alpha}}\alpha^f$, we conclude that $e^{f_{\Lambda,\alpha}}=e$. Thus, $f_{\Lambda,\alpha}\in \aut(C)$.
\end{proof}

Let $\Lambda_0\in \mathcal{L}$ and $\alpha_0\in \Lambda_0$. From Lemma~\ref{neigh} it follows that every $\Lambda\in \mathcal{L}$ with $\Lambda\neq \Lambda_0$ contains a unique $\alpha_\Lambda$ such that $(\alpha_0,\alpha_\Lambda)\in s_e$. 

\begin{lemm}\label{iso2}
With the above notation, $f_{\Lambda,\alpha_\Lambda}=f_{\Lambda_0,\alpha_0}$ for every $\Lambda\in \mathcal{L}$.
\end{lemm}

\begin{proof}
Let $\alpha=\alpha_\Lambda$. One can verify using the condition $s_e^f=s_e$ and $(\alpha_0,\alpha)\in s_e$ that
$$g_0g=e\Leftrightarrow g_0g\langle \alpha_0,\alpha \rangle=e\Leftrightarrow \langle g_0\alpha_0,g\alpha \rangle=e \Leftrightarrow$$ 
$$\Leftrightarrow (g_0\alpha_0,g\alpha)\in s_e\Leftrightarrow ((g_0\alpha_0)^f,(g\alpha)^f)\in s_e^f=s_e\Leftrightarrow$$
$$\Leftrightarrow\langle (g_0\alpha_0)^f,(g\alpha)^f \rangle=e\Leftrightarrow g_0^{f_{\Lambda_0,\alpha_0}}g^{f_{\Lambda,\alpha_\Lambda}}\langle \alpha_0^f,\alpha^f \rangle=e\Leftrightarrow g_0^{f_{\Lambda_0,\alpha_0}}g^{f_{\Lambda,\alpha_\Lambda}}=e$$
for all $g_0,g\in C$. By Lemma~\ref{iso1}, we have $f_{\Lambda_0,\alpha_0},f_{\Lambda,\alpha_\Lambda}\in \aut(C)$ and hence the above equality implies that 
$$g^{f_{\Lambda,\alpha_\Lambda}}=(g_0^{f_{\Lambda_0,\alpha_0}})^{-1}=(g_0^{-1})^{f_{\Lambda_0,\alpha_0}}=g^{f_{\Lambda_0,\alpha_0}}$$
for all $g_0,g\in C$. Thus, $f_{\Lambda,\alpha_\Lambda}=f_{\Lambda_0,\alpha_0}$. 
\end{proof}

Put $\varphi=f_{\Lambda_0,\alpha_0}$.

\begin{lemm}\label{iso3}
With the above notation, $\langle \alpha_{\Lambda_1}^f,\alpha_{\Lambda_2}^f \rangle=\langle \alpha_{\Lambda_1},\alpha_{\Lambda_2} \rangle^\varphi$ for all $\Lambda_1,\Lambda_2\in \mathcal{L}$.
\end{lemm}

\begin{proof}
One can verify straightforwardly that
$$g\langle \alpha_{\Lambda_1},\alpha_{\Lambda_2} \rangle=e\Leftrightarrow \langle g\alpha_{\Lambda_1},\alpha_{\Lambda_2} \rangle=e\Leftrightarrow (g\alpha_{\Lambda_1},\alpha_{\Lambda_2})\in s_e\Leftrightarrow((g\alpha_{\Lambda_1})^f,\alpha_{\Lambda_2}^f)\in s_e^f=s_e \Leftrightarrow$$
$$\Leftrightarrow \langle (g\alpha_{\Lambda_1})^f,\alpha_{\Lambda_2}^f \rangle=e\Leftrightarrow g^\varphi \langle \alpha_{\Lambda_1}^f,\alpha_{\Lambda_2}^f \rangle=e$$
for every $g\in C$, where the third equivalence follows from the condition $s_e^f=s_e$, whereas the fifth one from Lemma~\ref{iso2}. Therefore
$$\langle \alpha_{\Lambda_1}^f,\alpha_{\Lambda_2}^f \rangle=(g^\varphi)^{-1}=(g^{-1})^\varphi=\langle \alpha_{\Lambda_1},\alpha_{\Lambda_2} \rangle^\varphi,$$
where the second equality holds by Lemma~\ref{iso1}.
\end{proof}

Now let us return to the proof of the proposition. To prove the proposition, it is enough to show that $r_g^f=r_{g^{\varphi}}$ and $s_g^f=s_{g^{\varphi}}$ for every $g\in C$. Let $\alpha_1,\alpha_2\in \Omega$, $\Lambda_1,\Lambda_2\in \mathcal{L}$ such that $\alpha_1\in \Lambda_1$ and $\alpha_2\in \Lambda_2$, and $g_1,g_2\in C$ such that $\alpha_1=g_1\alpha_{\Lambda_1}$ and $\alpha_2=g_2\alpha_{\Lambda_2}$. It follows that
$$(\alpha_1,\alpha_2)\in r_g \Leftrightarrow \alpha_2=g\alpha_1\Leftrightarrow (\alpha_2)^f=(g\alpha_1)^f\Leftrightarrow (\alpha_2)^f=g^\varphi \alpha_1^f\Leftrightarrow (\alpha_1^f,\alpha_2^f)\in r_{g^\varphi},$$
where the third equivalence holds by Lemma~\ref{iso2}. So $r_g^f=r_{g^{\varphi}}$. Using Lemmas~\ref{iso1}-\ref{iso3}, one can verify that
$$(\alpha_1,\alpha_2)\in s_g \Leftrightarrow \langle \alpha_1,\alpha_2 \rangle=g \Leftrightarrow \langle g_1\alpha_{\Lambda_1}, g_2\alpha_{\Lambda_2}\rangle=g\Leftrightarrow g_1g_2\langle\alpha_{\Lambda_1},\alpha_{\Lambda_2}\rangle=g\Leftrightarrow$$
$$\Leftrightarrow g_1^\varphi g_2^\varphi\langle\alpha_{\Lambda_1},\alpha_{\Lambda_2}\rangle^\varphi=g^\varphi\Leftrightarrow g_1^{f_{\Lambda_1,\alpha_{\Lambda_1}}} g_2^{f_{\Lambda_2,\alpha_{\Lambda_2}}}\langle\alpha_{\Lambda_1}^f,\alpha_{\Lambda_2}^f\rangle=g^\varphi\Leftrightarrow $$
$$\Leftrightarrow \langle (g_1\alpha_{\Lambda_1})^f,(g_2\alpha_{\Lambda_2})^f\rangle=g^\varphi\Leftrightarrow \langle \alpha_1^f,\alpha_2^f\rangle=g^\varphi\Leftrightarrow (\alpha_1^f,\alpha_2^f)\in s_{g^\varphi}$$
and hence $s_g^f=s_{g^{\varphi}}$ as desired.
\end{proof}

Recall that $q=r^d$ for a prime $r$ and $d\geq 1$ and $\Sigma=\aut(\mathbb{F})$.  

\begin{lemm}\label{sigmaact}
If $d$ is odd, then $\Sigma^C\leq M$. 
\end{lemm}

\begin{proof}
Let $\sigma\in \Sigma$. Clearly, $\sigma^C\in \aut(C)$. Note that $|\sigma|$ divides $|\Sigma|=d$ which is odd. We conclude that $|\sigma^C|$ is odd and hence $\sigma^C\in M$ because $M$ is a Hall $2^\prime$-subgroup of $\aut(C)$. Thus, $\Sigma^C\leq M$.
\end{proof}

\begin{prop}\label{autdsrg}
If $d$ is odd, then $\aut(\Gamma(i,e))=(\GL(2,q)_K\rtimes \Sigma)^\Omega$ for each $i\in \{1,2\}$.
\end{prop}

\begin{proof}
Let $\Gamma=\Gamma(i,e)$. Clearly, $(s_e\cup r_{C_i})^f=s_e\cup r_{C_i}$ and hence 
$$(s_{C^\#}\cup r_{C_j})^f=(\Omega^2\setminus (1_\Omega \cup s_e\cup r_{C_i}))^f=\Omega^2\setminus (1_\Omega \cup s_e\cup r_{C_i})=s_{C^\#}\cup r_{C_j}$$
for every $f\in \aut(\Gamma)$, where $j\in \{1,2\}$ and $j\neq i$. As $s_e$ and $s_{C^\#}$ are symmetric, whereas $r_{C_1}$ and $r_{C_2}$ are nonsymmetric, we obtain $s_e^f=s_e$, $r_{C_1}^f=r_{C_1}$, and $r_{C_2}^f=r_{C_2}$. So $f$ satisfies the conditions of Proposition~\ref{isonorm} and consequently $f\in \iso(\mathcal{X}_0)$. Together with Proposition~\ref{isotatra}, this yields that 
$$\aut(\Gamma)\leq \iso(\mathcal{X}_0)=\GaL(2,q)^\Omega.$$

Let $f\in \GaL(2,q)^\Omega$. Then $f=(T\sigma)^\Omega$, where $T\in \GL(2,q)$ and $\sigma\in \Sigma$. From Lemma~\ref{actiongl} it follows that $s_e^f=s_{K\det(T)}$ and $r_{C_i}^f=r_{C_i^\sigma}$. The equality $s_e^f=s_e$ holds if and only if $\det(T)\in K$ and the equality $r_{C_i}^f=r_{C_i}$ holds by Lemma~\ref{sigmaact}. Thus, $\GaL(2,q)^\Omega\cap \aut(\Gamma)=(\GL(2,q)_K\rtimes \Sigma)^\Omega$. Together with $\aut(\Gamma)\leq \GaL(2,q)^\Omega$, this yields the required.
\end{proof}

\begin{rem2}
Let $g\in C$ and $T\in \GL(2,q)$ such that $\det(T)\in g$. Then the permutation $f_T$ induced by $T$ on $\Omega$ is an isomorphism from $\Gamma(e,i)$ to $\Gamma(g,i)$ for each $i\in\{1,2\}$ (Proposition~\ref{isodsrg}). Therefore $\aut(\Gamma(g,i))=f_T^{-1}\aut(\Gamma(e,i))f_T=f_T^{-1}(\GL(2,q)_K\rtimes \Sigma)^\Omega f_T$, where the latter equality holds by Proposition~\ref{autdsrg}.
\end{rem2}

Eqs.~\eqref{fieldker} and~\eqref{isopgl}, Proposition~\ref{autdsrg}, and Remark~2 imply the following statement.

\begin{corl}\label{autdsrgcorl}
If $d$ is odd, then $\aut(\Gamma(i,g))\cong \PGaL(2,q)$ for all $i\in \{1,2\}$ and $g\in C$.
\end{corl}

\begin{prop}\label{nonisodsrg}
If $d$ is odd, then the digraphs $\Gamma(1,h)$ and $\Gamma(2,g)$ are nonisomorphic for all $h,g\in C$.
\end{prop}

\begin{proof}
In view of Proposition~\ref{isodsrg}, it is enough to show that $\Gamma_1=\Gamma(1,e)$ and $\Gamma_2=\Gamma(2,e)$ are nonisomorphic. Assume the contrary. Then there is an isomorphism $f$ from $\Gamma_1$ to $\Gamma_2$. Clearly, $f$ maps the arc set of $\Gamma_1$ to the arc set of $\Gamma_2$. So 
$$(s_e\cup r_{C_1})^f=s_e\cup r_{C_2}$$ 
and hence 
$$(s_{C^\#}\cup r_{C_2})^f=(\Omega^2\setminus (1_\Omega \cup s_e\cup r_{C_1}))^f=\Omega^2\setminus (1_\Omega \cup s_e\cup r_{C_2})=s_{C^\#}\cup r_{C_1}.$$
Since $s_e$ and $s_{C^\#}$ are symmetric, whereas $r_{C_1}$ and $r_{C_2}$ are nonsymmetric, the above equalities yield that $s_e^f=s_e$ and 
\begin{equation}\label{rc1c2}
r_{C_1}^f=r_{C_2}~\text{and}~r_{C_2}^f=r_{C_1}.  
\end{equation}
Therefore $f\in \iso(\mathcal{X}_0)$ by Proposition~\ref{isonorm}.

Proposition~\ref{isotatra} implies that $f\in \GaL(2,q)^\Omega$ and hence $f=(T\sigma)^\Omega$ for some $T\in \GL(2,q)$ and $\sigma\in \Sigma$. From Lemma~\ref{actiongl} it follows that $r_g^f=r_{g^\sigma}$ for every $g\in C$. So $r_{C_1}^f=r_{C_1^\sigma}$. Due to Lemma~\ref{sigmaact}, we conclude that $\sigma^C\in M$. Since $C_1\in \orb(M,C)$, we obtain $C_1^\sigma=C_1$ and consequently $r_{C_1}^f=r_{C_1}$. However, $r_{C_1}^f=r_{C_2}$ by Eq.~\eqref{rc1c2}, a contradiction. 
\end{proof}

\begin{lemm}\label{dodd}
If $q-1=p(p-3)/4$, then $d$ is odd.
\end{lemm}

\begin{proof}
Assume the contrary that $d$ is even. Then $q-1=r^d-1=(r^{\frac{d}{2}}-1)(r^{\frac{d}{2}}+1)=p(p-3)/4$. So $p$ divides $r^{\frac{d}{2}}-1$ or $r^{\frac{d}{2}}+1$. In both cases, $r^{\frac{d}{2}}+1\geq p$. Therefore 
$$p(p-3)/4=(r^{\frac{d}{2}}-1)(r^{\frac{d}{2}}+1)\geq p(p-2),$$
a contradiction.
\end{proof}

Theorem~\ref{main1} immediately follows from Proposition~\ref{dsrg}, Corollary~\ref{autdsrgcorl}, Proposition~\ref{nonisodsrg}, and Lemma~\ref{dodd}.

\section{Divisible design graphs}

As in the previous section, we use the notation from Section~$3$. In this section, we provide a construction of DDGs and show that some of them are pairwise nonisomorphic.

Let $D$ be a difference set in $C$ (possibly, trivial) with parameters $(n,k,\lambda)$. Note that the relation 
$$s_D=\bigcup\limits_{g\in D} s_g$$
is symmetric because each $s_g$ so is. 

\begin{prop}\label{ddg}
The graph $\Delta=\Delta(D)=(\Omega,s_D)$ is a divisible design graph with parameters
\begin{equation}\label{ddgparam}
\left(n(q+1),kq,\lambda q,\frac{k^2(q-1)}{n},q+1,n\right).
\end{equation}
\end{prop}

\begin{proof}
In our notation, the matrices $E_0$, $E_1$, $E_2$, $E_3$ from~\cite[p.~227]{Reich} are equal to $A_e$, $A_{C^\#}=\sum\limits_{g\in C^\#} A_g$, $B_D=\sum\limits_{g\in D} B_g$, and $B_{C\setminus D}=\sum\limits_{g\in C\setminus D} B_g$, respectively. So the third equality from~\cite[p.~229]{Reich} implies that
\begin{equation}\label{sdsquare}
B_D^2=kqA_e+\lambda qA_{C^\#}+k^2mB_C.
\end{equation}
Observe that $B_D$ is the adjacency matrix of $\Delta$. Together with Lemma~\ref{ddg0}, the equality $m=\frac{q-1}{n}$, and Eqs.~\eqref{all1sum} and~\eqref{parab}, Eq.~\eqref{sdsquare} yields that $\Delta$ is a DDG with the required parameters.
\end{proof}

For the information on DSs in cyclic groups, we refer the readers to~\cite[Appendix,~Section~3]{BJL}. One can find among these DSs the ones of the following types: $(1)$ trivial; $(2)$ Singer; $(3)$ Paley of prime order; $(4)$ twin prime; $(5)$ biquadratic or octic residue. All of these DSs lead to DDGs with different parameters.

The statement below immediately follows from Proposition~\ref{ddg}.

\begin{corl}\label{srg}
The graph $\Delta(D)$ is a strongly regular graph if and only if $\lambda q=\frac{k^2(q-1)}{n}$. If the latter is the case, then $\Delta(D)$ has parameters
$$\left(n(q+1),kq,\lambda q,\lambda q\right).$$
\end{corl}

We can find only two families of DSs satisfying the condition $\lambda q=\frac{k^2(q-1)}{n}$. The first of them is the following. If $d$ is even, $n=r^{d/2}+1=\sqrt{q}+1$, and $D$ is a trivial DSs with parameters~$(n,n-1,n-2)$, then 
$$\lambda q=q(\sqrt{q}-1)=\frac{k^2(q-1)}{n}.$$
So the graph $\Delta(D)$ is an SRG with parameters
$$\left((q+1)(\sqrt{q}+1),q\sqrt{q},q(\sqrt{q}-1)\right).$$
The second one family is described further. Let $n=(r^d-1)/(r-1)$ and $D$ a DS with parameters 
$$\left(\frac{r^d-1}{r-1},r^{d-1},r^{d-2}(r-1)\right)$$
which are parameters of the complement to a Singer DS. Then
$$\lambda q=r^{2d-2}(r-1)=\frac{k^2(q-1)}{n}$$
and hence $\Delta(D)$ is an SRG with parameters
$$\left(\frac{r^{2d}-1}{r-1},r^{2d-1},r^{2d-2}(r-1),r^{2d-2}(r-1)\right).$$
In both cases, we obtain SRGs with parameters of the complement to the symplectic graph (see~\cite{BM,Kab}).

Since the edge set of $\Delta(D)$ is a union of some basic relations of $\mathcal{X}_0$, we have 
$$\aut(\Delta(D))\geq \aut(\mathcal{X}_0)=(\GL(2,q)_K\rtimes \Sigma_0)^\Omega.$$
Unfortunately, we can not say much on the number of pairwise nonisomorphic DDGs with the same parameters from Proposition~\ref{ddg}. One of the reasons for this is that $\aut(\Delta(D))$ can be much large than $\aut(\mathcal{X}_0)$. Nevertheless, we are able to prove the following statement concerned with the case when $n$ is prime.

\begin{prop}\label{isoddg}
Let $n$ be a prime. Suppose that difference sets $D_1$ and $D_2$ in $C$ are inequivalent and at least one of the divisible design graphs $\Delta(D_1)$ and $\Delta(D_2)$ is proper. Then $\Delta(D_1)$ and $\Delta(D_2)$ are nonisomorphic.
\end{prop}

\begin{proof}
If one of the DDGs $\Delta_1=\Delta(D_1)$, $\Delta_2=\Delta(D_2)$ is improper or $D_1$ and $D_2$ have distinct parameters, then it is clear that $\Delta_1$ and $\Delta_2$ are nonisomorphic. So we may assume that $\Delta_1$ and $\Delta_2$ are proper and $D_1$ and $D_2$ have the same parameters, say, $(n,k,\lambda)$. Suppose that $\Delta_1$ and $\Delta_2$ are isomorphic and $f$ is an isomorphism from $\Delta_1$ to $\Delta_2$.

To prove the lemma, is is enough to show that $D_1$ and $D_2$ are equivalent. Due to Eq.~\eqref{sdsquare}, the relation $r_{C^\#}$ consists of all pairs of vertices of $\Delta_i$, $i\in\{1,2\}$, that have $\lambda q$ common neighbors. Since each $\Delta_i$ is a proper DDG, we conclude that $r_{C^\#}^f=r_{C^\#}$ and hence $r_C^f=r_C$. Let $\Lambda_1,\Lambda_2\in \mathcal{L}$ with $\Lambda_1\neq\Lambda_2$ and $\alpha_i\in \Lambda_i$, $i\in\{1,2\}$. Lemma~\ref{isoscheme} implies that $\Lambda_i^f\in \mathcal{L}$ and the mapping $f_{\Lambda_i,\alpha_i}$ defined in Statement~$(2)$ of that lemma is a permutation on~$C$, $i\in \{1,2\}$. 

As $f$ is an isomorphism from $\Delta_1$ to $\Delta_2$, we have $s_{D_1}^f=s_{D_2}$. So
$$g_1g_2\langle \alpha_1,\alpha_2\rangle=\langle g_1\alpha_1,g_2\alpha_2\rangle \in D_1\Leftrightarrow g_1^{f_{\alpha_1,\Lambda_1}}g_2^{f_{\alpha_2,\Lambda_2}}\langle \alpha_1^f,\alpha_2^f\rangle=\langle (g_1\alpha_1)^f,(g_2\alpha_2)^f\rangle\in D_2$$
for all $g_1,g_2\in C$. The latter equality yields that
$$g_2\in D_1\langle \alpha_1,\alpha_2\rangle^{-1}g_1^{-1}\Leftrightarrow g_2^{f_{\alpha_2,\Lambda_2}}\in D_2\langle \alpha_1^f,\alpha_2^f\rangle^{-1}(g_1^{f_{\alpha_1,\Lambda_1}})^{-1}.$$
Therefore the bijections $\varphi_1:C\rightarrow C$ and $\varphi_2:\mathcal{B}_1\rightarrow \mathcal{B}_2$, where $\mathcal{B}_1=\{D_1g:~g\in C\}$ and $\mathcal{B}_2=\{D_2g:~g\in C\}$, such that 
$$g^{\varphi_1}=g^{f_{\alpha_2,\Lambda_2}}$$
and
$$(D_1g)^{\varphi_2}=D_2\langle \alpha_1^f,\alpha_2^f\rangle^{-1}((\langle \alpha_1,\alpha_2\rangle g^{-1})^{f_{\alpha_1,\Lambda_1}})^{-1}$$
satisfy the condition
$$g\in D_1h\Leftrightarrow g^{\varphi_1}\in (D_1h)^{\varphi_2}.$$
Thus, the designs $\dev(D_1)=(C,\mathcal{B}_1)$ and $\dev(D_2)=(C,\mathcal{B}_2)$ are isomorphic. Since $|C|=n$ is prime, $D_1$ and $D_2$ are equivalent by Lemma~\ref{palfy}.
\end{proof}

Several inequivalent DSs with Singer parameters in a cyclic group of prime order are known (see, e.g.,~\cite{Gordon} and~\cite[Section~VI.17]{BJL}). If $n=2^i-1$ is a Mersenne prime, then there exist at least $\varphi(i)/2=(i-1)/2$ pairwise inequivalent DSs in $C$ with Singer parameters~$(2^i-1,2^{i-1}-1,2^{i-2}-1)$ by~\cite{DD}. 

We would like to finish the paper with the following question.

\begin{ques}
How many pairwise nonisomorphic divisible design graphs (strongly regular graphs, resp.) with the same parameters come from Proposition~\ref{ddg} (Corollary~\ref{srg}, resp.)?
\end{ques}

\begin{rem3}
S. Goryanov has noted in private communication that the divisible design graphs (strongly regular graphs, resp.) from Proposition~\ref{ddg} (Corollary~\ref{srg}, resp.) were independently developed in a different context in joint work with B. De Bruyn and W. Yan, and that a paper describing these results will appear.
\end{rem3}

\vspace{5mm}

\noindent \textbf{Acknowledgements:} The second author was supported by the state contract of the Sobolev Institute of Mathematics (project number FWNF-2026-0011).


\begin{thebibliography}{list}

\bibitem{BH}
\emph{P. T. Bateman, R. A. Horn}, A heuristic asymptotic formula concerning the distribution of prime numbers. Math. Comp., \textbf{16} (1962), 220--228.



\bibitem{BJL}
\emph{T.~Beth, D.~Jungnickel, H.~Lenz}, Design Theory, 2nd edition, Cambridge University Press, Cambridge (1999).

\bibitem{BCN}
\emph{A.~Brouwer, A.~Cohen, A.~Neumaier}, Distance-regular graphs, Springer, Heidelberg (1989).

\bibitem{BCS}
\emph{A. E. Brouwer, D. Crnkovi\'{c}, A. $\check{S}$vob},  A construction of directed strongly regular graphs with parameters $(63,11,8,1,2)$, Discrete Math., \textbf{347} (2024), Article ID 114146.


\bibitem{BCSZ}
\emph{A. E. Brouwer, D. Crnkovi\'{c}, A. $\check{S}$vob, M.~Zubovi\'{c} $\check{Z}$utolija},  Some directed strongly regular graphs constructed from
linear groups,  Appl. Algebra Eng. Commun. Comput.,  https://doi.org/10.1007/s00200-025-00703-8.


\bibitem{BH}
\emph{A. E. Brouwer, S. A. Hobart}, Parameters of directed strongly regular graphs, http://homepages .cwi .nl /~aeb /math /dsrg /dsrg .html.


\bibitem{BM}
\emph{A. E. Brouwer, H. Van Maldeghem}, Strongly Regular Graphs, Cambridge Univ. Press (2022).


\bibitem{Cam1}
\emph{P.~Cameron}, Permutation groups, in: Handbook of Combinatorics, Vol.~1, eds. R. L. Graham et al.,  Elsevier-The MIT Press (1995), 611--645.



\bibitem{Cam2}
\emph{P.J. Cameron, H.R. Maimani, G.R. Omidi, B. Tayfeh-Rezaie}, $3$-designs from $\PSL(2,q)$, Discrete Math., \textbf{306} (2006), 3063--3073.




\bibitem{Cam3}
\emph{P.J. Cameron, G.R. Omidi, B. Tayfeh-Rezaie}, $3$-designs from $\PGL(2,q)$,  Electron. J. Comb., \textbf{13} (2006), Article number R50.



\bibitem{CP}
\emph{G.~Chen, I.~Ponomarenko}, Coherent configurations, Central China Normal University Press, Wuhan (2019).




\bibitem{Dickson}
\emph{L.E. Dickson}, Cyclotomy, higher congruences, and Waring's problem, Am. J. Math. \textbf{57} (1935) 391--424.


\bibitem{DD}
\emph{J.~F.~Dillon, H.~Dobbertin}, New cyclic difference sets with Singer parameters, Finite Fields Their Appl., \textbf{10} (2004), 342--389.



\bibitem{DF}
\emph{B. Drabkin, C. French}, On a class of non-commutative imprimitive association schemes of rank~$6$, Commun. Algebra, \textbf{43}, No.~9 (2015), 4008--4041.




\bibitem{Duv}
\emph{A.~M. Duval}, A directed graph version of strongly regular graphs, J. Comb. Theory, Ser. A, \textbf{47} (1988) 71--100.



\bibitem{EP}
\emph{S.~Evdokimov,~I.~Ponomarenko}, Schurity of $S$-rings over a cyclic group and generalized wreath product of permutation groups, St. Petersburg Math. J., \textbf{24}, No.~3 (2013), 431--460. 







\bibitem{Gordon}
\emph{D.~Gordon}, Database of difference sets, https://www.dmgordon.org/diffset/.


\bibitem{HKM}
\emph{W. H. Haemers, H. Kharaghani, M. A. Meulenberg}, Divisible design graphs,  J. Combin. Theory, Ser. A, \textbf{118} (2011), 978--992.


\bibitem{Kab}
\emph{V.~Kabanov}, A new construction of strongly regular graphs with parameters of the complement symplectic graph, Electron. J. Comb., \textbf{30} (2023), Article number P1.25.



\bibitem{KhS}
\emph{H.~Kharaghani, S.~Suda}, Non-commutative association schemes and their fusion association schemes, Finite Fields Their Appl., \textbf{52} (2018), 108--125.


\bibitem{GAP}
\emph{M. Klin, C. Pech, S. Reichard}, COCO2P -- a GAP package, 0.14, 07.02.2015, http://www.math.tu-dresden.de/~pech/COCO2P.


\bibitem{JZ}
\emph{G.~Jones, A.~Zvonkin}, Block designs, permutation groups and prime values of polynomials, Trudy Inst. Matem. Mekh. UrO Ran, \textbf{29}, No,~1 (2023), 233--253.


\bibitem{Palfy}
\emph{P.~P.~P\'alfy}, Isomorphism problem for relational structures with a cyclic automorphism, Eur. J. Comb., \textbf{8} (1987), 35--43. 




\bibitem{Reich}
\emph{S.~Reichard}, Tatra schemes and their mergings, in: Isomorphisms, symmetry and computations in algebraic graph theory, eds. G. A. Jones et al.,  Springer Proceedings in Mathematics and Statistics, \textbf{305} (2020), 219--234.


\bibitem{Ry} 
\emph{G.~Ryabov}, On separability of Schur rings over abelian $p$-groups, Algebra Log., \textbf{57}, No.~1 (2018), 49--68.

\bibitem{Schur}
\emph{I.~Schur}, Zur theorie der einfach transitiven Permutationgruppen, S.-B. Preus Akad. Wiss. Phys.-Math. Kl., \textbf{18}, No.~20 (1933), 598--623.


\bibitem{Wi}
\emph{H.~Wielandt}, Finite permutation groups, Academic Press, New York-London (1964).

\end{thebibliography}
\end{document}